\documentclass[10pt]{amsart}
\usepackage{pdfpages}

\usepackage{graphicx,cancel,xcolor}
\usepackage{amsmath,amsthm,amssymb}
\usepackage{a4wide}
\usepackage[utf8]{inputenc}
\usepackage{graphicx}
\usepackage{comment}
\usepackage{dsfont}
\usepackage{subfigure}%
\usepackage{float}
\makeatletter
\@namedef{r@tocindent4}{0pt}
\@namedef{r@tocindent5}{0pt}
\makeatother
\makeatletter
\renewcommand{\tocsection}[3]{%
	\indentlabel{\@ifnotempty{#2}{\bfseries\ignorespaces#1 #2\quad}}\bfseries#3}
\renewcommand{\tocsubsection}[3]{%
	\indentlabel{\@ifnotempty{#2}{\ignorespaces#1 #2\quad}}#3}
\newcommand\@dotsep{4.5}
\def\@tocline#1#2#3#4#5#6#7{\relax
	\ifnum #1>\c@tocdepth 
	\else
	\par \addpenalty\@secpenalty\addvspace{#2}%
	\begingroup \hyphenpenalty\@M
	\@ifempty{#4}{%
		\@tempdima\csname r@tocindent\number#1\endcsname\relax
	}{%
		\@tempdima#4\relax
	}%
	\parindent\z@ \leftskip#3\relax \advance\leftskip\@tempdima\relax
	\rightskip\@pnumwidth plus1em \parfillskip-\@pnumwidth
	#5\leavevmode\hskip-\@tempdima{#6}\nobreak
	\leaders\hbox{$\m@th\mkern \@dotsep mu\hbox{.}\mkern \@dotsep mu$}\hfill
	\nobreak
	\hbox to\@pnumwidth{\@tocpagenum{\ifnum#1=1\bfseries\fi#7}}\par
	\nobreak
	\endgroup
	\fi}
\AtBeginDocument{%
	\expandafter\renewcommand\csname r@tocindent0\endcsname{0pt}
}
\def\l@subsection{\@tocline{2}{0pt}{2.5pc}{5pc}{}}
\def\l@subsubsection{\@tocline{2}{0pt}{4.5pc}{5pc}{}}
\def\l@paragraph{\@tocline{2}{0pt}{5.5pc}{5pc}{}}
\makeatother

\def\rr{\mathbb R}

\def\HH{\mathcal H}
\def\ZZ{\mathcal Z}

\def\AA{\mathcal A}

\def\OO{\mathcal O}

\def\LL{\mathcal L}
\def\OO{\mathcal O}

\def\NN{\mathcal N}

\def\prl{\parallel }
\def\prl{\parallel }

\newcommand{\set}[1]{\{#1\}}
\providecommand{\abs}[1]{\lvert#1\rvert}
\providecommand{\norm}[1]{\lVert#1\rVert}

\newtheorem{theorem}{Theorem}[section]
\newtheorem{lemma}[theorem]{Lemma}

\theoremstyle{definition}
\newtheorem{defi}{Definition}

\theoremstyle{remark}
\newtheorem{remark}{Remark}

\numberwithin{equation}{section}

\numberwithin{equation}{section}

\begin{document}
\thispagestyle{empty}

\title[Controllability and stabilization of coupled wave equations]{Exact controllability and stabilization of locally \\coupled wave equations : theoretical results}
\author{St\'ephane Gerbi}
\address{Laboratoire de Math\'ematiques UMR 5127 CNRS \& Universit\'e de Savoie Mont Blanc, Campus scientifique, 73376 Le Bourget du Lac Cedex, France}\email{stephane.gerbi@univ-smb.fr}
\author{Chiraz Kassem}
\address{Universit\'e Libanaise,
			EDST, Equipe EDP-AN,
			Hadath, Beirut, Lebanon}
			\email{shiraz.kassem@hotmail.com}
			\author{Amina Mortada}
			\address{ Universit\'e Libanaise,
			EDST, Equipe EDP-AN,
			Hadath, Beirut, Lebanon}
			\email{amina$\_$mortada2010@hotmail.com}
			\author{Ali Wehbe}
			\address{Universit\'e Libanaise,
			Facult\'e des Sciences 1,
			EDST, Equipe EDP-AN,
			Hadath, Beirut, Lebanon}
				\email{ali.wehbe@ul.edu.lb}
	 
	\date{}
	
	\begin{abstract}
	In this paper, we study the  exact controllability and stabilization of a system of two wave equations coupled by velocities with an internal, local control acting on only one equation. We distinguish two cases. In the first one, when the waves propagate at the same speed: using a frequency domain approach combined with multiplier technique, we prove that the system is exponentially stable when the coupling region  is a subset of the damping region and satisfies the geometric control condition GCC (see Definition \ref{def1} below). Following a result of Haraux (\cite{Haraux1989}), we establish the main indirect observability inequality. This results leads, by the HUM method, to prove that the total system is exactly controllable by means of locally distributed control. In the second case, when the waves propagate at different speed, we establish an exponential decay rate in the weak energy space under appropriate geometric conditions. Consequently, the system is exactly controllable using a result of \cite{Haraux1989}.    	
	\end{abstract}
	
	\subjclass[2010]{35L10, 35B40, 93D15, 90D20}
	\keywords{Coupled wave equations, internal damping, exact controllability}
	
 \maketitle
\tableofcontents
\section{Introduction} \label{S1l}
\subsection{Motivation and aims}
Let $\Omega$ be a nonempty connected open subset of $\rr^N$ having a boundary $\Gamma$ of class $C^2$. In  \cite{Alabau2017}, F. Alabau et al. considered the energy decay of a system of two wave equations coupled by velocities
 \begin{eqnarray}
 u_{tt} - a\Delta u + \rho(x,u_t) + b(x)y_t &=&  0  \hskip 1.4 cm \mbox{in} \,\,\, \Omega \times \rr_+^{*},\label{e111.2}\\
 y_{tt}- \Delta y - b(x)u_t &=&  0  \hskip 1.4 cm \mbox{in}\,\,\, \Omega \times \rr_+^{*},\label{e122.1}\\
 u = y &=& 0 \hskip 1.4 cm \mbox{on} \,\, \,  \Gamma\times \rr_+^{*},\label{e133.2}
 \end{eqnarray}
 where $a>0$ constant, $b\in C^0(\overline \Omega, \mathbb{R})$ and $\rho(x,u_t)$ is a non linear damping.
  Using an approach based on multiplier techniques, weighted nonlinear inequalities and the optimal-weight convexity method (developed in \cite{AlabauConvexity2005}), the authors established an explicit energy decay formula in terms of the behavior of the nonlinear feedback close to the origin. Their results are obtained in the case when the following three conditions are satisfied: the waves propagate at the same speed ($a=1$),  the coupling coefficient $b(x)$ is small and positive ($0\leq b(x)\leq b_0$, $b_0\in (0,b^\star]$ where $b^\star$ is a constant depending on $\Omega$ and on the control region) and both the coupling and the damping regions satisfying an appropriated  geometric conditions named Piecewise Multipliers Geometric Conditions (introduced in \cite{Liu1997}, recalled below in Definition \ref{def2} and denoted by PMGC in short). In their work, the case where the waves are not assumed to be propagated with equal speeds ($a$ is not necessarily equal to 1) and/or the coupling coefficient $b(x)$ is not assumed to be positive and small has been left as an open problem even when the damping term $\rho$ is linear with respect to the second variable. Recently, C. Kassem et al. in \cite{Wehbe-Amina-Chiraz}, answered this important open question by studying the stabilization of the following linear system:    
 \begin{eqnarray}
 u_{tt}-a\Delta u+ c(x) u_t + b(x)y_t &=& 0 \hskip 1.4 cm \mbox{in} \,\,\, \Omega \times \rr_+^{*},\label{e11.2}\\
 y_{tt}- \Delta y - b(x)u_t &=&  0  \hskip 1.4 cm \mbox{in}\,\,\, \Omega \times \rr_+^{*},\label{e12.2}\\
 u = y &=& 0 \hskip 1.4 cm \mbox{on} \,\, \,  \Gamma\times \rr_+^{*},\label{e13.2}
 \end{eqnarray}
in the case where the waves propagate with equal or different speeds and the coupling coefficient is not assumed to be positive and small. Indeed, they distinguished two cases. The first one is when the waves propagate at the same speed (i.e. $a=1$), but unlike the works of \cite{Alabau2017}, the coupling coefficient function $b$ is not necessarily assumed to be positive and small. In this case, under the condition that the coupling region and the damping region have non empty intersection satisfying the PMGC conditions, they established an exponential energy decay rate for weak initial data. On the contrary (i.e. $a\not=1$ ) they first proved the lack of the exponential stability of the system. However, under the same geometric condition, an optimal energy decay rate of type $\frac{1}{t}$ was established for smooth initial data. Notice that, the PMGC conditions is much more restrictive than the Geometric Control Condition (introduced in \cite{Bardos-Lebeau-Rauch}, recalled below in Definition \ref{def1} and, denoted by GCC in short). The natural question is then whether or not stabilization and the exact controllability still hold in the case where the coupling region and the damping region have non empty intersection satisfying the GCC condition?\\
The aim of this paper is to investigate the exact controllability of the following system:
\begin{eqnarray}
u_{tt}-a\Delta u + b(x)y_t &=&  c(x)  v_t \hskip 0.7 cm \mbox{in} \,\,\, \Omega \times \rr_+^{*},\label{e11.2c}\\
y_{tt}- \Delta y - b(x)u_t &=&  0  \hskip 1.4 cm \mbox{in}\,\,\, \Omega \times \rr_+^{*},\label{e12.2c}\\
u = y &=& 0 \hskip 1.4 cm \mbox{on} \,\, \,  \Gamma\times \rr_+^{*},\label{e13.2c}
\end{eqnarray}
with the following initial data
\begin{equation}
u(x,0)=u_0, \, \, \, y(x,0)=y_0,  \, \, \, u_t(x,0)=u_1 \, \, {\rm and} \, \, y_t(x,0)=y_1, \quad x\in \Omega,
\end{equation}
under appropriate geometric conditions. Here, $a>0$ constant, $b\in C^0(\overline \Omega, \mathbb{R})$, $c \in C^0(\overline \Omega, \mathbb{R}^+)$ and $v$ is an appropriate control. 
The idea is to use a result of A. Haraux in \cite{Haraux1989} for which the observability of the homogeneous system associated to  \eqref{e11.2c}-\eqref{e13.2c} is equivalent to the exponential stability of system \eqref{e11.2}-\eqref{e13.2} in an appropriate Hilbert space. So, we provide a complete analysis for the exponential stability of system \eqref{e11.2}-\eqref{e13.2} in different Hilbert spaces. First, when the waves propagate at the same speed (i.e., $a=1$), under the condition that the coupling region is included in the damping region and satisfies the so-called Geometric Control Condition (GCC in Short), we establish the exponential stability of system \eqref{e11.2}-\eqref{e13.2}. Consequently, an observability inequality of the solution of the homogeneous system associated to \eqref{e11.2c}-\eqref{e13.2c} in the space $\left(H^1_0(\Omega)\times L^2(\Omega)\right)^2$ is established. This leads, by the HUM method introduced by Lions in \cite{Lions88}, to the exact controllability of system \eqref{e11.2c}-\eqref{e13.2c} in the space $\left(H^{-1}(\Omega)\times L^2(\Omega)\right)^2$. Noting that, the geometric situations covered here are richer than those considered in \cite{Alabau2017} and \cite{Wehbe-Amina-Chiraz}. Furthermore, on the contrary when the waves propagate at different speeds, (i.e., $a \neq 1$), we establish the exponential stability of system \eqref{e11.2}-\eqref{e13.2} in the space $H^{1}_0(\Omega)\times L^2(\Omega)\times L^2(\Omega)\times H^{-1}(\Omega)$ provided that the damping region satisfies the PMGC condition while the coupling region includes in the damping region and satisfying the GCC conditions. Consequently, an observability inequality of the solution of the homogeneous system associated to \eqref{e11.2c}-\eqref{e13.2c} is established.  This leads, by the HUM method, to the exact controllability of system \eqref{e11.2c}-\eqref{e13.2c} in the space $L^2(\Omega)\times H^{-1}(\Omega)\times H^{1}_0(\Omega)\times L^2(\Omega)$.

\subsection{Literature}
Since the work of J. L. Lions in \cite{Lions88}, the observability and controllability of coupled wave equations have been studied by an intensive number of publications. In \cite{Lions88}, J. L. Lions studied the complete and partial observability and controllability of coupled systems of either hyperbolic-hyperbolic type or hyperbolic-parabolic type. These results assume that the coupling parameter is sufficiently small. In \cite{Alabau01} and \cite{Alabau2003}, F. Alabau studied the indirect boundary observability of an abstract system of two weakly coupled second order evolution equations where the coupling coefficient is strictly positive in the whole domain. In particular, using a piecewise multiplier method, she proved that, for a sufficiently large time $T$, the observation of the trace of the normal derivative of the first component of the solution on a part of the boundary allows us to get back a weakened energy of the initial data. Consequently, using Hilbert Uniqueness Method, she proved that the system is exactly controllable for small coupling parameter by means of one boundary control. Noting that, the situation where the waves propagate with different speeds is not covered. Later, the indirect boundary controllability of a system of two weakly coupled one-dimensional wave equations has been studied by Z. Liu and B. Rao in \cite{LiuRao2009}. Using the non harmonic analysis, they established several weak observability inequalities which depend on the ratio of the wave propagation speeds and proved the indirect exact controllability. The null controllability of the reaction diffusion System has been studied by F. Ammar-Khodja et al. in \cite{Ammarkhodja}, by deriving an observability estimate for the linearized problem. The exact controllability of a system of weakly coupled wave equations with an internal locally control acted on only one equation has been studied by A. Wehbe and W. Youssef in \cite{Wehbe-Youssef-2010} and \cite{Wehbe-Youssef-2011}. They showed that, for sufficiently large time, the observation of the velocity of the first component of the solution on a neighborhood of a part of the boundary allows us to get back a weakened energy of initial data of the second component, this if the coupling parameter is sufficiently small, but non- vanishing and by the HUM method, they proved that the total system is exactly controllable. F. Alabau and M. L\'eautaud in \cite{AlabauLeautaudStabi:2013}, considered a symmetric systems of two wave-type equation, where only one of them being controlled. The two equations are coupled by zero order terms, localized in part of the domain. They obtained an internal and a boundary controllability result in any space dimension, provided that both the coupling and the control regions satisfy the Geometric Control Condition. 
\\
\subsection{Description of the paper}
This paper is organized as follows: In section \eqref{section2}, first, we show that the system \eqref{e11.2}-\eqref{e13.2} can be reformulated into a first order evolution equation and we deduce the well posedness property of the problem by the semigroup approach. Second, by using Theorem $2.2$ of \cite{Wehbe-Amina-Chiraz}, we show that our problem is strongly stable without geometric conditions. In section \ref{section3}, we show the exponential decay rate of system \eqref{e11.2}-\eqref{e13.2} when the coupling region $b$ is a subset of the damping region $c$ and satisfies the geometric control condition GCC. After that, we show that our system is exactly controllable by using Proposition $2$ of A. Haraux in \cite{Haraux1989}. In section \ref{section4}, we show the exponential decay rate of system \eqref{e11.2}-\eqref{e13.2} in the weak energy space provided that the damping region satisfies the PMGC condition while the coupling region is a subset of the damping region and satisfies the GCC condition. 
\section{Well posedeness and strong stability}\label{section2}
Let us define the energy space $\HH = \bigg(H_0^1(\Omega) \times L^2(\Omega)\bigg)^2$ equipped with the following inner product and norm, respectively : for all $U=(u,v,y,z), \, \widetilde{U}=(\widetilde{u},\widetilde{v},\widetilde{y},\widetilde{z})\in \HH ,$
\begin{equation*}
(U, \,\, \widetilde{U})_\HH\,=
\,\displaystyle\,a\int_{\Omega}(\nabla u\cdot\nabla
\widetilde{u})dx\,+\,\int_{\Omega}v\widetilde{v}\,dx+\int_{\Omega}(\nabla y\cdot\nabla
\widetilde{y})dx\,+\int_{\Omega}z\widetilde{z}dx, \, \left\| U\right\|_{\mathcal{H}} = \sqrt{(U,U)_{\mathcal{H}}},
\end{equation*}

Let $(u,u_t,y,y_t)$ be a regular solution of the system \eqref{e11.2}-\eqref{e13.2}. Its associated energy is defined by 
\begin{equation*}
E(t)=\frac{1}{2} \int_\Omega \left(\abs{u_{t}}^2+a\abs{\nabla u}^2+\abs{y_{t}}^2+\abs{\nabla y}^2\right) dx. \label{energy}
\end{equation*}
A straightforward computations gives  
$E'(t) = \displaystyle- \int_{\Omega}c(x) |u_t|^2dx \leq 0.$
Consequently, system \eqref{e11.2}-\eqref{e13.2} is dissipative in the sense that its energy is non-increasing with respect to $t$.
Setting $U=(u,u_t,y,y_t)$,  system \eqref{e11.2}-\eqref{e13.2} may be recast as: 
$$
U_t= \AA U, \; \hbox{in} \; (0,+\infty), U(0)=(u_0,u_1,y_0,y_1),$$
 where the unbounded operator $\AA : D(\AA) \subset \HH \rightarrow \HH$ is given by:
\begin{align}
D(\AA)&=\bigg((H^2(\Omega) \cap H_0^1(\Omega)) \times
H_0^1(\Omega)\bigg)^2 \mbox{ and }\label{Da}\\
\AA U&=(\,v , a \Delta u - b z- c v, \,z,\Delta y+b v\,),
\displaystyle \hskip 1 cm \forall\, \,  U\,=\,(u ,v, y , z)\,
\in\,D(\AA).\label{A}
\end{align}
{Note that due to the fact that $c(x) \geq 0$, the operator $\AA$ is dissipative in $\HH$. And, by applying the Lax-Milgam Theroem, it is easy to prove that the operator $\AA$ is maximal in $\HH$ i.e.  $R( I -\AA)=\HH$. Consequently, it  generates a $C_0$-semigroup of contractions $(e^{t\AA})_{t \geq 0}$.} So, system \eqref{e11.2}-\eqref{e13.2} is wellposed in $\HH$.\\ 
We need now to study the asymptotic behavior of $E(t)$. For this aim, we suppose that there exists a non empty open $\omega_{c_+}\subset \Omega$ satisfying the following condition 
$$  \set{x\in \Omega :c(x)  > 0}\supset\overline{\omega}_{c_+}. \hskip 2 cm {\rm (LH1)}$$
On the other hand, 
as $b(x)$ is not identically null and continuous, then there exists a non empty open 
$\omega_{b}\subset \Omega$ such that 
$$ \set{x\in \Omega :b(x)  \neq 0}\supset\overline{\omega}_{b} .\hskip 1. cm \mbox{(LH2)}$$
If $\omega=\omega_{c_+}\cap\omega_{b} \neq \emptyset$  and condition (LH1) holds, then system \eqref{e11.2}-\eqref{e13.2} is strongly stable using Theorem 2.2 in \cite{Wehbe-Amina-Chiraz}, i.e.
$$\displaystyle \lim_{t \rightarrow +\infty}\norm{e^{t\AA}(u_0,u_1,y_0,y_1)}_\HH = 0 \quad \forall (u_0,u_1,y_0,y_1) \in \HH. $$
\section{Exponential stability and exact controllability in the case $a=1$}\label{section3}
\subsection{Exponential stability}
This subsection is devoted to study the exponential stability of system \eqref{e11.2}-\eqref{e13.2} in the case when the waves propagate at the same speed, i.e., $a=1$ under an appropriate geometric conditions. 
Before we state our results, we recall the Geometric Control Conditions GCC introduced by Rauch and Taylor in \cite{Rauch1974} for manifolds without boundaries and by Bardos, Lebeau and Rauch in \cite{Bardos-Lebeau-Rauch} for domains with boundaries and the Piecewise Multipliers Geometric Condition introduced by K. Liu in \cite{Liu1997}. 
\begin{defi}
We say that a subset $\omega$ of $\Omega$ satisfies the $\textbf{GCC}$ if every ray of the geometrical optics starting at any point $x \in \Omega$ at $t=0$ enters the region $\omega$ in finite time $T.$ \label{def1}
\end{defi}

\begin{defi}
We say that $\omega$ satisfies the Piecewise Multipliers Geometric Condition (PMGC in short) if there exist $\Omega_j\subset \Omega$ having Lipschitz boundary $\Gamma_j=\partial \Omega_j$ and $x_j\in \mathbb{R}^N$, $j=1,...,J$
such that $\Omega_j\cap \Omega_i= \emptyset$ for $j\not=i$ and  
$\omega$ contains a neighborhood in $\Omega$ of the set $\displaystyle\cup _{j=1}^{J} \gamma_j\left(x_j\right) \cup \left( \Omega \setminus \displaystyle \cup _{j=1}^{J} \Omega_j \right)$ where 
$\gamma_j(x_j) = \{ x \in  \Gamma_j:(x-x_j)\cdot \nu_j(x) > 0 \}$ and $\nu_j$ is the outward unit normal vector to $ \Gamma_j$. \label{def2}
\end{defi}
\begin{remark}
The PMGC is the generalization of the Multipliers Geometric Condition (MGC in short) introduced by Lions in \cite{Lions88}, saying that $\omega$ contains a neighborhood in $\Omega$ of the set
$ \{ x \in  \Gamma :(x-x_0)\cdot \nu(x) > 0 \}$, for some $x_0\in \mathbb{R}^N$, where $\nu$ is the outward unit normal vector to $ \Gamma =\partial \Omega$.
\end{remark}
Now, we are in position to state our first main result by the following theorem : 
\begin{theorem}\label{ThexpE.2} (Exponential decay rate)
Let $a=1$. Assume that conditions {\rm (LH1)} and {\rm (LH2)} hold. Assume also that $\omega_{b}\subset\omega_{c_+}$ satisfies the  geometric control conditions  {\rm GCC} and that $b,c\in W^{1,\infty}(\Omega)$. Then there exist positive constants $M\geq 1$, $\theta >0$
such that for all initial data $(u_0,u_1,y_0,y_1)\in \HH$ the energy
of the system (\ref{e11.2})-(\ref{e13.2}) satisfies the following decay rate:
\begin{equation}
E(t)\leq Me^{-\theta t}E(0),\ \ \ \ \forall t>0.\label{ExpSta1.2}
\end{equation}
\end{theorem}
\begin{remark}
The geometric situations covered by Theorem \ref{ThexpE.2} are richer than those considered in \cite{Wehbe-Amina-Chiraz} and \cite{Alabau2017}. Indeed, in the previous references, the authors consider the PMGC geometric conditions that are more restrictive than GCC. On the other hand, unlike the results in \cite{Alabau2017}, we have no restriction in Theorem \ref{ThexpE.2} on the upper bound and the sign of the coupling function coefficient $b$. This theorem is then a generalization in the linear case of the result of  \cite{Alabau2017} where the coupling coefficient considered have to satisfy $0\leq b(x)\leq b_0$, $b_0\in (0,b^\star]$ where $b^\star$ is a constant depending on $\Omega$ and on the control region.
\end{remark}
In order to prove Theorem \ref{ThexpE.2}, we apply a result of Huang \cite{Huang-85} and Pr\"uss \cite{Pruss-84}.
A $C_0$- semigroup of contraction $(e^{t\AA})_{t\geqslant 0}$ in a Hilbert space $\HH$ is exponentially stable if and only if the two following hypotheses are fulfilled:
\begin{align}
i \rr
\subseteq \rho(\AA)       &\label{H1}\tag{H1}\\
\limsup_{\beta \, \in \rr, |\beta| \rightarrow + \infty}\parallel (i \beta I - \AA )^{-1} \parallel_{\mathcal{L}(\mathcal{H})} <
\infty &\label{H2}\tag{H2}
\end{align}
Since the resolvent of $\AA$ is compact and $0 \in \rho(\AA)$, then from the fact that our system is strongly stable, we deduce that condition \eqref{H1} is satisfied. We now prove that condition \eqref{H2} holds, using an argument of contradiction. For this aim, we suppose that there exist a  real sequence $\beta_n $ with $\beta_n \rightarrow +\infty$ and a sequence $U_n=(u_n,v_n,y_n,z_n) \in D(\AA)$ such that 	
\begin{align}
\parallel U_n \parallel_{\HH} &= 1 \quad \mbox{ and }\label{E1}\\
\lim_{n\rightarrow\infty}\parallel  (i \beta_n I -\AA) U_n\parallel_{\HH} &= 0.\label{E2}
\end{align}
Next, detailing equation (\ref{E2}), we get
\begin{eqnarray}
	i\beta_n u_n-v_n &=& f_n^1 \,\,\rightarrow \,0 \hskip 0.5 cm \mbox{in}\hskip 0.5 cm H_0^1(\Omega),\label{e46}\\
	i\beta_n v_n- \Delta u_n +b z_n + cv_n&=& g_n^1 \,\,\rightarrow \,0 \hskip 0.5 cm \mbox{in}\hskip 0.5 cm L^2(\Omega),\label{e47}\\
	i\beta_n y_n-z_n &=& f_n^2 \,\,\rightarrow \,0 \hskip 0.5 cm \mbox{in}\hskip 0.5 cm H_0^1(\Omega),\label{e48}\\
	i\beta_n z_n- \Delta y_n - b v_n  &=& g_n^2 \,\,\rightarrow
	\,0 \hskip 0.5 cm \mbox{in}\hskip 0.5 cm L^2(\Omega).\label{e49}
\end{eqnarray}
Eliminating $v_n$ and $z_n$ from the previous system, we obtain the following system
\begin{equation}
	\beta_n^2 u_n +\Delta
	u_n-i\beta_n b y_n-i\beta_n c u_n=-g_n^1-bf_n^2
	-i\beta_nf_n^1-cf_n^1, \label{E5}
\end{equation}
\begin{equation}
	\beta_n^2 y_n + \Delta y_n +i \beta_n b u_n = -i \beta_n
	f_n^2 + b f_n^1 -g_n^2. \label{E4}
\end{equation}
On the other side, we notice that $v_n$ and $z_n$ are uniformly
bounded in $L^2(\Omega)$. It follows, from  equations  (\ref{e46}) and (\ref{e48}), that 
\begin{equation}
	\displaystyle \int_{\Omega} |y_n|^2 dx =\frac{O(1)}{\beta_n^2}\ \ \ \textrm{and}\ \ \
	\int_{\Omega} |u_n|^2 dx =\frac{O(1)}{\beta_n^2}.\label{e411.3} 
\end{equation}
For clarity, we divide the proof into several Lemmas.
\begin{lemma}
	The solution $(u_n,v_n,y_n,z_n) \in D(\AA)$ of system (\ref{e46})-(\ref{e49}) satisfies the following estimates
	\begin{equation}
		\displaystyle\int_{\Omega} c |\beta_n u_n|^2dx=o(1) \quad and \quad \int_{\omega_{c_+}} |\beta_nu_n|^2 dx = o(1).\label{E33}
	\end{equation}
\end{lemma}
\begin{proof}
	First, since $U_n$ is uniformly bounded in $\HH$, then from \eqref{E2}, we get 
	\begin{equation}
	\mathrm{Re}\left\{i\beta_n \parallel U_n \parallel^2_\HH - (\AA U_n,
		U_n)_\HH \right\}= \int_{\Omega}c(x)|v_n|^2 dx=o(1).\label{E3}
	\end{equation}
 Under condition (LH1), it follows that
	\begin{equation}
		\int_{\omega_{c_+}}|v_n|^2 dx =o(1).
	\end{equation}
	Then, using equations \eqref{E3} and \eqref{e46},  we get
	\begin{equation}
		\int_{\Omega}c|\beta_n u_n|^2 dx =o(1).
	\end{equation}
Consequently, we have
	\begin{equation*}
		\int_{\omega_{c_+}}|\beta_n u_n|^2 dx =o(1). 
	\end{equation*}
	The proof is thus complete.
{$\quad\square$}\end{proof}
\begin{lemma}
	The solution $(u_n,v_n,y_n,z_n) \in D(\AA)$ of system (\ref{e46})-(\ref{e49})  satisfies the following estimates
	\begin{equation}
\displaystyle\int_{\Omega} c |\nabla u_n|^2 dx = o(1) \quad and \quad \int_{\omega_{c_+}} |\nabla u_n|^2 dx = o(1).\label{E31.2} 
\end{equation}	
\end{lemma}
\begin{proof}
Multiplying equation \eqref{E5} by $c \overline{u}_n$, integrating by parts and using the fact that $u_n=0$ on $\Gamma$, we get
\begin{align}\label{equation0.2}
	&\int_{\Omega}c|\beta_nu_n|^2dx -\int_{\Omega}c|\nabla u_n|^2dx-\int_{\Omega}(\nabla c \cdot \nabla u_n)\overline u_ndx -i\int_{\Omega} \beta_n b y_n c \overline u_ndx \nonumber\\
	& -i \int_{ \Omega}\beta_n c u_n \overline u_ndx= \int_{\Omega}(-g_n^1-bf_n^2
	-i\beta_nf_n^1-cf_n^1)c \overline u_n dx.
\end{align}
Using the fact that $f_n^1$, $f_n^2$ converge to zero in $H_0^1(\Omega)$, $g_n^1$ converges to zero in $L^2(\Omega)$ and $\beta_n \overline u_n$ is uniformly bounded in $L^2(\Omega)$, we obtain
\begin{equation}\label{equation1.2}
	\int_{\Omega}(-g_n^1-bf_n^2
	-i\beta_nf_n^1-cf_n^1)c \overline u_n dx=o(1).
\end{equation}
Using the fact that $\nabla u_n$, $\beta_ny_n$, $\beta_nu_n$ are uniformly bounded in $L^2(\Omega)$ and $\norm{u_n}=o(1)$, we get
\begin{equation}\label{equation2.2}
	-\int_{\Omega}(\nabla c \cdot \nabla u_n)\overline u_ndx -i\int_{\Omega} \beta_n b y_n c \overline u_ndx 
	-i \int_{ \Omega}\beta_n c u_n \overline u_ndx=o(1).
\end{equation}
Inserting \eqref{equation1.2} and \eqref{equation2.2} into \eqref{equation0.2}, we get 
\begin{equation} \label{equation3.2}
	\int_{\Omega}c|\beta_nu_n|^2dx -\int_{\Omega}c|\nabla u_n|^2dx=o(1).
\end{equation}
Finally, using estimation \eqref{E33} in \eqref{equation3.2}, we deduce 
\begin{equation*}
	\int_{\omega_{c_+}} |\nabla u_n|^2 dx = o(1).
\end{equation*}
The proof is thus complete.
{$\quad\square$}\end{proof}
\begin{lemma}
The solution $(u_n,v_n,y_n,z_n) \in D(\AA)$ of system (\ref{e46})-(\ref{e49})  satisfies the following estimate
\begin{equation}\label{nablay0.2}
\displaystyle \int_{\omega_{b}} |\nabla y_n|^2 dx = o(1).
\end{equation}	
\end{lemma}
\begin{proof}
The proof contains three points.\\
i) First, multiplying equation \eqref{E5} by $\frac{1}{\beta_n} \Delta \overline y_n$, then using Green's formula and the fact that $ u_n= f_n^1=0$ on $\Gamma$, we obtain	
\begin{align}\label{nablay.2}
&\displaystyle -\int_{\Omega}  \beta_n ( \nabla u_n \cdot \nabla \overline y_n)dx  + \frac{1}{\beta_n}\int_{\Omega}  \Delta u_n \Delta \overline y_ndx + i \int_{\Omega}  (\nabla b \cdot \nabla \overline y_n)y_n  dx \nonumber\\\nonumber \\
&+ i \int_\Omega b |\nabla y_n|^2dx+ i \int_{\Omega} (\nabla c \cdot \nabla \overline y_n)u_ndx + i \int_\Omega c(\nabla u_n \cdot \nabla \overline y_n)dx \\ \nonumber \\
&=\displaystyle \int_{\Omega}(-g_n^1 -b f_n^2 -cf_n^1)\frac{1}{\beta_n}  \Delta \overline y_ndx + i \int_\Omega  (\nabla f_n^1 \cdot \nabla \overline y_n)dx. \nonumber
\end{align}	
As $f_n^1$, $f_n^2$ converge to zero in $H_0^1(\Omega)$, $g_n^1$ converges to zero in $L^2(\Omega)$ and the fact that $\frac{1}{\beta_n}  \Delta  y_n$, $\nabla y_n$ are uniformly bounded in $L^2(\Omega)$, we have
\begin{equation}\label{nablay1.2}
\int_{\Omega}(-g_n^1 -b f_n^2 -cf_n^1)\frac{1}{\beta_n}  \Delta \overline y_ndx + i \int_\Omega  (\nabla f_n^1 \cdot \nabla \overline y_n)dx=o(1).
\end{equation}
Using the fact that $\nabla y_n$ is uniformly bounded in $L^2(\Omega)$, $\norm{u_n}_{L^2(\Omega)} =o(1)$, $\norm{y_n}_{L^2(\Omega)} = o(1)$ and using the estimation \eqref{E31.2}, we get 
\begin{equation}
 i \int_{\Omega}  (\nabla b \cdot \nabla \overline y_n)y_n  dx  
+ i \int_{\Omega} (\nabla c \cdot \nabla \overline y_n)u_ndx + i \int_\Omega c(\nabla u_n \cdot \nabla \overline y_n)dx = o(1). \label{nablay2.2}
\end{equation}	
Inserting now \eqref{nablay1.2} and	\eqref{nablay2.2} into \eqref{nablay.2}, we get 
\begin{equation}\label{nablayy.2}
\displaystyle-\int_{\Omega}  \beta_n  (\nabla u_n \cdot \nabla\overline y_n)dx + \frac{1}{\beta_n}\int_{\Omega}  \Delta u_n \Delta \overline y_ndx + i \int_\Omega b |\nabla y_n|^2dx =o(1).
\end{equation}
ii) Similarly, multiplying equation \eqref{E4} by $\frac{1}{\beta_n}  \Delta \overline u_n$, then using Green's formula and the fact that $ y_n= f_n^2=0$ on $\Gamma$, we obtain
\begin{align}
&\displaystyle -\int_{\Omega}  \beta_n  (\nabla y_n \cdot \nabla\overline u_n)dx  + \frac{1}{\beta_n}\int_{\Omega} \Delta y_n \Delta \overline u_ndx 
-i \int_{\Omega} \, (\nabla b \cdot \nabla \overline u_n)u_n  dx \nonumber\\ 
&- i \int_\Omega b |\nabla u_n|^2dx
=\displaystyle \int_{\Omega}(b f_n^1 - g_n^2)\frac{1}{\beta_n}  \Delta \overline u_ndx +i \int_\Omega  (\nabla f_n^2 \cdot \nabla \overline u_n)dx. \label{nablayu.2}
\end{align}	
Using the fact that $f_n^1$, $f_n^2$ converge to zero in $ H_0^1(\Omega)$, $ g_n^2$ converges to zero in $L^2(\Omega)$ and the fact that 
$\frac{1}{\beta_n} \Delta  u_n$, $\nabla u_n$ are  uniformly bounded in $L^2(\Omega)$, we get
\begin{equation}\label{nablayu1.2}
 \int_{\Omega}(b f_n^1 - g_n^2)\frac{1}{\beta_n}  \Delta \overline u_ndx + i \int_\Omega  (\nabla f_n^2 \cdot \nabla \overline u_n)dx=o(1).
\end{equation} 
Also, using the fact that $\nabla u_n$ is uniformly bounded in $L^2(\Omega)$, $\norm{u_n}_{L^2(\Omega)} =o(1)$, we have
\begin{equation}\label{nablayu2.2}
	-i \int_{\Omega}  (\nabla b \cdot \nabla \overline u_n)u_n  dx   = o(1).
\end{equation}
Inserting \eqref{nablayu1.2} and  \eqref{nablayu2.2} into  \eqref{nablayu.2}, we get	
\begin{equation}
	-\int_{\Omega}  \beta_n  (\nabla y_n \cdot \nabla\overline u_n)dx +\frac{1}{\beta_n}\int_{\Omega}  \Delta y_n \Delta \overline u_ndx - i \int_\Omega b|\nabla u_n|^2dx =o(1).\label{nablay3.2}
\end{equation}

iii) Finally, by combining \eqref{nablayy.2} and\eqref{nablay3.2} and taking the imaginary part, we obtain
\begin{equation}\label{equation4.2}
	\int_{\Omega} b  |\nabla y_n|^2dx =   \int_\Omega b|\nabla u_n|^2dx + o(1).
\end{equation}
Since $\omega_b \subset \omega_{c^{+}}$, it follows from \eqref{E31.2} and \eqref{equation4.2} that 
$$\int_{\omega_{b}}  |\nabla y_n|^2dx = o(1).$$
The proof is thus complete.
{$\quad\square$}\end{proof}
\begin{lemma}
The solution $(u_n,v_n,y_n,z_n) \in D(\AA)$ of system (\ref{e46})-(\ref{e49})  satisfies the following estimate
\begin{equation}
\displaystyle \int_{\omega_{b}} |\beta_n y_n|^2 dx = o(1).\label{betay.2}
\end{equation}	
\end{lemma}
\begin{proof}
Multiplying equation \eqref{E4} by $b \overline y_n$. Then using Green's formula and the fact that $y_n=0$ on $\Gamma$, we obtain  
\begin{align}\label{betay0.2}
&\int_{\Omega}b |\beta_n y_n|^2dx -\int_{\Omega} b |\nabla y_n|^2dx - \int_{\Omega}(\nabla b \cdot \nabla y_n)  \overline y_n dx\nonumber\\
&+ i \int_{\Omega} b^2\,\beta_n u_n \overline y_ndx = \int_{\Omega}(-i \beta_n f_n^2 + bf_n^1 - g_n^2) b  \overline y_n dx.
\end{align}
As $f_n^1$, $f_n^2$ converge to zero in $H_0^1(\Omega)$, $g_n^2$ converges to zero in $L^2(\Omega)$ and $\beta_n y_n$ is uniformly bounded in $L^2(\Omega)$, we get
\begin{equation}\label{betay1.2}
\int_{\Omega}(-i \beta_n f_n^2 + bf_n^1 - g_n^2) b  \overline y_n dx = o(1).
\end{equation}  
Using the fact that $\beta_n u_n$ and $\nabla y_n$ are uniformly bounded in $L^2(\Omega)$ and $\norm {y_n}_{L^2(\Omega)} =o(1)$, we get 
\begin{equation}\label{betay2.2}
\int_{\Omega}(\nabla b \cdot \nabla y_n)  \overline y_n
  + i \int_{\Omega} b^2 \,\beta_n u_n \overline y_ndx = o(1).
\end{equation}
Inserting \eqref{betay1.2}, \eqref{betay2.2} into \eqref{betay0.2}, we obtain
\begin{equation*}
\int_{\Omega}b|\beta_n y_n|^2dx -\int_{\Omega} b |\nabla y_n|^2dx=o(1).
\end{equation*}
Using the estimation \eqref{nablay0.2} in the previous equation, we get 
\begin{equation*}
\int_{\Omega}b |\beta_n y_n|^2dx=o(1).
\end{equation*}
This yields
\begin{equation*}
\int_{\omega_{b}} |\beta_n y_n|^2dx=o(1).
\end{equation*}
The proof is thus complete.
{$\quad\square$}\end{proof}
\begin{lemma}\label{auxiliary.2}
Let $f_n$ be a bounded sequence in $L^2(\Omega)$. Then the solution $\phi_n \in H_0^1(\Omega) \cap H^2(\Omega)$ of the following system
\begin{equation}\label{auxpro12.2}
\left\{
\begin{matrix}
\beta^2_n \phi_n +\Delta \phi_n -i b\beta_n \phi_n &=& f_n &\mbox{in} & \Omega ,\\
\phi_n&=&0 &\mbox{on} & \Gamma ,\\
\end{matrix}
\right.
\end{equation}
verifies the following estimate
\begin{equation}\label{maj2.2}
\int_{\Omega}(|\beta_n \phi_n|^2 + |\nabla \phi_n|^2)dx \leq C \int_{\Omega} |f_n|^2dx,
\end{equation} 
where $C$ is a constant independent of $n$. 	
\end{lemma}
\begin{proof}
Consider the following wave equation
\begin{equation}\label{auxpro2.2}
\left \{
\begin{matrix}
\phi_{tt} - \Delta \phi + b \phi_t &=&0 &\rm{in} & \Omega,\\
\phi \quad& = &0 &\rm{on}& \Gamma.
\end{matrix}
\right.
\end{equation}	
System \eqref{auxpro2.2} is wellposed in the space $H = H_0^1(\Omega) \times L^2(\Omega)$ and since $\omega_{b}$ verifies GCC condition then it is exponentially stable (see \cite{Bardos-Lebeau-Rauch}). Therefore, following Huang \cite{Huang-85} and Pr\"uss \cite{Pruss-84}, we deduce that the resolvent of its corresponding operator 
$$\AA_{aux} : D(\AA_{aux}) \longrightarrow H_0^1(\Omega) \times L^2(\Omega)$$
defined by 
$D(\AA_{aux}) =(H^2(\Omega) \cap H_0^1(\Omega)) \times H_0^1(\Omega) \quad \mbox{and} \quad \AA_{aux}(\phi, \tilde \phi)= (\tilde \phi, \Delta \phi - b \tilde \phi)$
is uniformly bounded on the imaginary axis. 

On the other hand, system \eqref{auxpro12.2} can be rewritten in the form:
\begin{equation}
\left\{
\begin{matrix}
i\beta_n \phi_n - \tilde \phi_n&=& 0, \\ 
i\beta_n \tilde \phi_n - \Delta \phi_n + b \tilde \phi_n&=&-f_n. 
\end{matrix}
\right.
\end{equation}
So, 
\begin{equation}
\begin{pmatrix}
i\beta_n - \AA_{aux}
\end{pmatrix}
\begin{pmatrix}
\phi_n \\
\tilde \phi_n
\end{pmatrix}
=
\begin{pmatrix}
0 \\
-f_n
\end{pmatrix}.
\end{equation}
Equivalently, 
\begin{equation}
\begin{pmatrix}
\phi_n \\
\tilde \phi_n
\end{pmatrix}
=
\begin{pmatrix}
i\beta_n - \AA_{aux}
\end{pmatrix}^{-1}
\begin{pmatrix}
0 \\
-f_n
\end{pmatrix}.
\end{equation}
This yields 
\begin{align}
\norm{(\phi_n, \tilde \phi_n)}^2_{H} & \leq \norm {(i\beta_n - \AA_{aux})^{-1}}^2_{\LL{(H)}} \norm{(0,-f_n)}^2_{H}\nonumber\\
&\leq C \int_{\Omega}|f_n|^2dx,
\end{align}
where $C$ is a constant independent of $n$.
Consequently, we deduce
\begin{equation*}
\int_{\Omega}(|\beta_n \phi_n|^2 + |\nabla \phi_n|^2)dx \leq C \int_{\Omega} |f_n|^2dx.
\end{equation*} 
The proof is thus complete.
{$\quad\square$}\end{proof}
\begin{lemma}
The solution $(u_n,v_n,y_n,z_n) \in D(\AA)$ of system (\ref{e46})-(\ref{e49})  satisfies the following estimate
\begin{equation}
\displaystyle \int_{\Omega} |\beta_n u_n|^2 dx = o(1).\label{betau.2}
\end{equation}	
\end{lemma}
\begin{proof} Taking $f_n = u_n$ in Lemma \ref{auxiliary.2} and  
multiplying equation \eqref{E5} by $\beta_n^2 \overline  \phi_n$ where $\phi_n$ is a solution of  \eqref{auxpro12.2}. Then using Green's formula and the fact that $u_n = \phi_n = 0$ on $\Gamma$, we obtain
\begin{equation} \label{Eau1.2}
\begin{array}{l}
\displaystyle
\int_{\Omega} \beta_n^2u_n(\beta_n^2 \overline \phi_n  + \Delta \overline \phi_n)dx - i  \int_{\Omega}b\beta_n y_n \beta^2_n\overline\phi_n dx - i \int_{\Omega}c \beta_n u_n \beta_n^2\overline\phi_ndx\\ \\
\displaystyle =\int_{\Omega}(-g_n^1-bf_n^2-cf_n^1)\beta_n^2\overline \phi _ndx -i \int_{\Omega} \beta_n f_n^1 \beta_n^2 \overline \phi_n dx.
\end{array}
\end{equation}  

Substituting the first equation of system \eqref{auxpro12.2} into the first term of \eqref{Eau1.2}, we get 
\begin{equation} \label{Eau11.2}
\begin{array}{l}
\displaystyle
\int_{\Omega} |\beta_nu_n|^2dx -i\int_{\Omega} \beta_n^2\overline \phi_n b \beta_n u_ndx- i  \int_{\Omega}b\beta_n y_n \beta^2_n\overline\phi_n dx - i \int_{\Omega}c \beta_n u_n \beta_n^2\overline\phi_ndx\\ \\
\displaystyle =\int_{\Omega}(-g_n^1-bf_n^2-cf_n^1)\beta_n^2\overline \phi _ndx -i \int_{\Omega} \beta_n f_n^1 \beta_n^2 \overline \phi_n dx.
\end{array}
\end{equation}
As $f_n^1$, $f_n^2$ converge to zero in $H_0^1(\Omega)$ and $\beta_n^2 \phi_n$ is uniformly bounded in $L^2(\Omega)$ due to \eqref{maj2.2}, we get
\begin{equation}\label{auxe2.2}
\int_{\Omega}(-g_n^1-bf_n^2-cf_n^1)\beta_n^2\overline \phi _ndx = o(1).
\end{equation}
From the first equation of \eqref{auxpro12.2}, we have $\beta_n^2 \overline \phi_n = \overline u_n - \Delta \overline \phi_n -i b \beta_n \overline \phi_n$.  Consequently, we have
\begin{align}
 - i\int_{\Omega}\beta_nf_n^1 \beta_n^2 \overline \varphi_{n}dx &= - i\int_{\Omega}\beta_nf_n^1(\overline u_n - \Delta \overline \phi_n -i b\beta_n \overline \phi_n)dx \nonumber \\
 &= -i \int_{ \Omega}\beta_nf_n^1\overline u_ndx -i\int_{ \Omega}\beta_n (\nabla \overline \phi_n \cdot \nabla f_n^1)dx - \int_{\Omega}b f_n^1 \beta_n^2 \overline\phi_n,
\end{align} 
which yields
\begin{equation}\label{auxe1.2}
- i\int_{\Omega}\beta_nf_n^1 \beta_n^2 \overline \varphi_{n}dx = o(1),
\end{equation}
because $f_n^1$ converges to zero in $H_0^1(\Omega)$ and $\beta_n u_n$, $\beta_n^2 \phi_n$, $\beta_n\nabla \phi_n$ are uniformly bounded in $L^2(\Omega)$. 

Substituting now \eqref{auxe2.2} and \eqref{auxe1.2} into \eqref{Eau11.2}
\begin{equation}
\int_{\Omega} |\beta_nu_n|^2dx --i\int_{\Omega} \beta_n^2\overline \phi_n b\beta_n u_ndx- i  \int_{\Omega}b\beta_n y_n \beta^2_n\overline\phi_n dx - i \int_{\Omega}c \beta_n u_n \beta_n^2\overline\phi_ndx=o(1).
\end{equation}
Finally, using estimations \eqref{E33},
\eqref{betay.2} and the fact that $\beta_n^2 \overline \phi_n$ is uniformly bounded in $L^2(\Omega)$ into the previous equation, we obtain 
\begin{equation}
\int_{\Omega} |\beta_nu_n|^2dx=o(1).
\end{equation}
The proof is thus complete.
{$\quad\square$}\end{proof}
\begin{lemma}
The solution $(u_n,v_n,y_n,z_n) \in D(\AA)$ of system (\ref{e46})-(\ref{e49})  satisfies the following estimate
\begin{equation}
\displaystyle \int_{\Omega} |\beta_n y_n|^2 dx = o(1).\label{betayO.2}
\end{equation}	
\end{lemma}
\begin{proof}
Taking $f_n = y_n$ in Lemma \ref{auxiliary.2}. Multiplying equation \eqref{E4} by $\beta_n^2 \overline  \phi_n$ where $\phi_n$ is a solution of  \eqref{auxpro12.2}. Then using Green's formula and the fact that $y_n = \phi_n=0$ on $\Gamma$, we obtain

\begin{equation} \label{Eau7.2}
\begin{array}{l}
\displaystyle
		\int_{\Omega} \beta_n^2y_n(\beta_n^2 \overline \phi _n + \Delta \overline \phi_n)dx + i  \int_{\Omega}b \beta_n u_n \beta_n^2 \overline\phi_ndx \\ \\
		\displaystyle = - i\int_{\Omega}\beta_nf_n^2 \beta_n^2\overline \phi _n+ \int_{\Omega}(bf_n^1-g_n^2)\beta_n^2\overline \phi _ndx.
	\end{array}
\end{equation}  
Then, substituting the first equation of problem \eqref{auxpro12.2} into the first term of \eqref{Eau7.2}, we get
\begin{equation}
	\begin{array}{l}\label{Eaux7}
		\displaystyle
		\int_{\Omega}|\beta_n y_n|^2dx -i\int_{\Omega}b\beta_n^2\overline \phi_n \beta_n y_ndx+ i  \int_{\Omega}b \beta_n u_n \beta_n^2 \overline\phi_ndx \\ \\
		\displaystyle = - i\int_{\Omega}\beta_nf_n^2 \beta_n^2\overline \phi _n+ \int_{\Omega}(bf_n^1-g_n^2)\beta_n^2\overline \phi _ndx.
	\end{array}
\end{equation} 
Since $\beta_n^2 \overline \phi_n$ is uniformly bounded in $L^2(\Omega)$, $f_n^1$ converges to zero in $H_0^1(\Omega)$ and $g_n^2$ converges to zero in $L^2(\Omega)$, we have
\begin{equation}\label{Eau8}
	\int_{\Omega}(-bf_n^1-g_n^2)\beta_n^2\overline \phi _ndx= o(1).
\end{equation} 
Moreover, using the first equation of problem \eqref{auxpro12.2} and integrating by parts yields
  
\begin{align}\label{Eau9}
\displaystyle-i\int_{\Omega}\beta_nf_n^2 \beta_n^2 \overline \phi_ndx &= -i\int_{\Omega}(\overline y_n-\Delta \overline \phi_n - i b \beta_n \overline \phi_n )\beta_n f_n^2dx \nonumber\\ 
\displaystyle&= -i \int_{\Omega} f_n^2 \beta_n \overline y_ndx -i\int_{\Omega} \beta_n (\nabla \overline \phi_n.\nabla f_n^2)dx - \int_{\Omega} b f_n^2 \beta_n^2 \overline \phi_n dx.
\end{align}    
Using the fact that $\beta_n  y_n$, $\beta_n^2  \phi_n$ and $\beta_n \nabla \phi_n$ are uniformly bounded in $L^2(\Omega)$  and $f_n^2$ converges to zero in $H_0^1(\Omega)$ in \eqref{Eau9}, we get 
\begin{equation}\label{Eau10}
-i	\int_{\Omega}\beta_nf_n^2 \beta_n^2 \overline \phi_ndx =o(1).
\end{equation}
Inserting \eqref{Eau8}, \eqref{Eau10} into \eqref{Eaux7}, we obtain  
\begin{equation}\label{Eau1}
\int_{\Omega}|\beta_n y_n|^2dx -i\int_{\Omega}b\beta_n^2\overline \phi_n \beta_n y_ndx+ i  \int_{\Omega}b \beta_n u_n \beta_n^2 \overline\phi_ndx = o(1).
\end{equation}
Finally, using \eqref{betay.2}, \eqref{betau.2} and the fact that $\beta_n^2 \overline\phi_n$ is uniformly bounded in $L^2(\Omega)$, we deduce
\begin{equation*}
\int_{\Omega} |\beta_n y_n|^2 dx = o(1).
\end{equation*}
The proof is thus complete.
{$\quad\square$}\end{proof}
\begin{lemma}
The solution $(u_n,v_n,y_n,z_n) \in D(\AA)$ of system (\ref{e46})-(\ref{e49})  satisfies the following estimates
\begin{equation}\label{graduyo.2}
\int_{\Omega} |\nabla u_n|^2 dx = o(1) \quad {\rm{and}} \quad \int_{\Omega} |\nabla y_n|^2 dx = o(1).	
\end{equation}	
\end{lemma}
\begin{proof}
 Multiplying equation \eqref{E5} by $\overline u_n$, applying Green's formula and using the fact that $u_n=0$ on $\Gamma$, we get
\begin{equation}\label{Eau13}
	\int_{\Omega}|\beta_n u_n|^2dx -\int_{\Omega} |\nabla u_n|^2dx-i\int_{\Omega} \beta_n b y_n \overline u_ndx - i \int_{ \Omega} \beta_n c u_n \overline u_ndx = o(1).
\end{equation} 
Using the fact that $\beta_nu_n$, $\beta_ny_n$ are uniformly bounded in $L^2(\Omega)$, $\norm {u_n}=o(1) $  and  the estimation \eqref{betau.2}  in \eqref{Eau13}, we obtain    
\begin{equation}\label{Eau14}
	\int_{\Omega} |\nabla u_n|^2dx=o(1).
\end{equation}  
Similarly, multiplying equation \eqref{E4} by $\overline y_n$ and applying Green's formula and using the fact that $y_n=0$ on $\Gamma$, we get
\begin{equation}\label{Eau15}
	\int_{\Omega}|\beta_n y_n|^2dx -\int_{\Omega} |\nabla y_n|^2dx + i\int_{\Omega} \beta_n b u_n \overline y_ndx=o(1).
\end{equation}  
Using the fact that $\beta_nu_n$ is uniformly bounded in $L^2(\Omega)$, $\norm {y_n}=o(1) $  and \eqref{betayO.2} in \eqref{Eau15}, we obtain    
\begin{equation}\label{Eau16}
	\int_{\Omega} |\nabla y_n|^2dx=o(1).
\end{equation}  
The proof is thus complete.
{$\quad\square$}\end{proof}
\textbf{Proof of Theorem \ref{ThexpE.2}. }
It follows from \eqref{betau.2}, \eqref{betayO.2} and \eqref{graduyo.2}
that $\norm{U_n}_{\HH}=o(1)$ which is a contradiction with \eqref{E1}. Consequently, condition (H2) holds and the energy of system \eqref{e11.2}-\eqref{e13.2} decays exponentially to zero.
The proof is thus complete.{$\quad\square$}
\subsection{ Observability and exact controllability}
\label{S4.2}
First, we consider the following homogeneous system associated to \eqref{e11.2c}-\eqref{e13.2c} for $a=1$ by:
\begin{eqnarray}
\psi_{tt}-\Delta \psi + b(x)\varphi_t&=&  0  \hskip 1.4 cm \mbox{in} \,\,\, \Omega \times \rr_+,\label{eh11.2}\\
\varphi_{tt}- \Delta \varphi - b(x)\psi_t &=&  0  \hskip 1.4 cm \mbox{in}\,\,\, \Omega \times \rr_+,\label{eh12.2}\\
\psi = \varphi &=& 0 \hskip 1.4 cm \mbox{on} \,\, \,  \Gamma\times \rr_+,\label{eh13.2} \\
\psi(\cdot,0) = \psi_0, \psi_t (\cdot,0) = \psi_1, \varphi(\cdot,0) &=& \varphi_0, \varphi_t (\cdot,0) = \varphi_1 \hskip 1.4 cm \mbox{in} \,\, \,  \Omega.\label{eh13.2bis}
\end{eqnarray}
Let $ \Phi=(\psi,\psi_t, \varphi, \varphi_t)$ be a regular solution of system \eqref{eh11.2}-\eqref{eh13.2}, its associated total energy is given by:  
\begin{equation}
E(t)=\frac{1}{2} \int_\Omega \left(\abs{\psi_{t}}^2+\abs{\nabla \psi}^2+\abs{\varphi_{t}}^2+\abs{\nabla \varphi}^2\right) dx. \label{energy.2}
\end{equation}
A direct computation gives 
\begin{equation}
\frac{d}{dt}E(t)=0.
\end{equation}
Thus, system \eqref{eh11.2}-\eqref{eh13.2} is conservative in the sense that its energy $E(t)$ is constant. It is also wellposed and admits a unique solution in the energy space $\HH$.

Now, we establish the direct and indirect inequality given by the following theorem: 
\begin{theorem}\label{IO.2}
Let $a=1$. Assume that conditions {\rm{(LH1)}} and {\rm (LH2)} hold. Assume also that $\omega_{b} \subset \omega_{c_+}$  satisfies the geometric control condition $\rm{GCC}$ and that $b$, $c \in W^{1,\infty} (\Omega)$. Then there exists a time $T_0$ such that for all $T>T_0$, there exist two constants $M_1>0$, $M_2>0$ such that the solution of system \eqref{eh11.2}-\eqref{eh13.2} satisfies the following observability inequalities: 
\begin{equation}\label{IO}
M_1 \norm{\Phi_0}^2_{\HH} \leq \int_0^T \int_\Omega c(x) \, |\psi_t|^2dxdt \leq M_2 \norm{\Phi_0}^2_{\HH}, 
\end{equation}
for all $\Phi_0 = (\psi_0,\psi_1,\varphi_0, \varphi_1) \in \HH$. 
\end{theorem}
\begin{proof}
The direct inequality follows from the definition of the total energy for all $T>0$. While the proof of the inverse inequality is a direct consequence of Proposition 2 of Haraux in \cite{Haraux1989} for which the exponentially stability \eqref{ExpSta1.2} implies the existence of a time $T_0>0$ such that for all $T>T_0$ there exist two constants $M_1>0$ and $M_2>0$ such that \eqref{IO} holds. The proof is thus complete. 
{$\quad\square$}\end{proof}

%
%
Now, we are ready to study the exact controllability of a system \eqref{e11.2c}-\eqref{e13.2c} by using the HUM method.
First, thanks to the direct inequality, the solution of the system of equations \eqref{e11.2c}, \eqref{e12.2c}, \eqref{e13.2c} can be obtained as usual by the method of transposition (see \cite{Lions88} and \cite{Lions1988}).
Let $v_0 \in L^2(0,T;L^2(w_{c_+}))$, we choose the control
\begin{equation}
v(t)=-\frac{d}{dt}v_0(t) \in [H^1(0,T;L^2(\omega_{c_+}))]',
\end{equation}
where the derivative $\frac{d}{dt}$ is not taken in the sense of distributions but in the sense of the duality $H^1(0,T;L^2(\omega_{c_+}))$ and its dual $[H^1(0,T;L^2(\omega_{c_+}))]'$, i.e.,
$$-\int_0^T \frac{d}{dt} v_1(t) \mu (t) dt= \int_0^T v_1(t) \frac{d}{dt} \mu(t) dt, \quad \forall \mu \in H^1(0,T;L^2(\omega_{c_+}).$$ 
Then we have the followig result:
\begin{theorem}
	\label{t2}
	Let $T > 0$ and $a=1$. Assume that conditions {\rm{(LH1)}} and {\rm{(LH1)}} hold. Assume also that $\omega_{b} \subset \omega{c_+}$  satisfies the geometric control condition {\rm{GCC}} and that $b$, $c \in W^{1,\infty} (\Omega)$.
	 Given	$$U_0=(u_0,u_1,y_0,y_1)\,  \in  (L^2(\Omega) \times  H^{-1}(\Omega))^2, \, \, \, v = -\frac{d}{dt} v_0 \in [ H^1(0,T;L^2(\omega_{c_+}))]', $$  the controlled system \eqref{e11.2c}-\eqref{e13.2c}   has a unique weak solution
	$$U =(u,u_t,y,y_t)\in C^0 ( [ 0 , T ] , (L^2(\Omega) \times  H^{-1}(\Omega))^2).$$
\end{theorem}
\begin{proof}
	Let $(\psi,\psi_t,\varphi,\varphi_t)$ be the solution of \eqref{eh11.2}-\eqref{eh13.2}  associated to  $\Phi_0 = (\psi_0,\psi_1,\varphi_0,\varphi_1)$. Multiplying the first equation of \eqref{e11.2c}-\eqref{e13.2c}  by $\psi$ and the second by $\varphi$ and integrating by parts, we obtain 
	\begin{equation}
	\left\lbrace
	\begin{array}{l}
	\displaystyle
	\int_\Omega y_t(T)\varphi(T)dx+\int_\Omega u_t(T)\psi(T) dx-
	\int_\Omega y(T)\varphi_t(T) dx \\-\displaystyle  \int_\Omega u(T)\psi_t(T) dx
	-\displaystyle\int_\Omega bu(T)\varphi(T) dx+\int_\Omega by(T )\psi(T)dx=
	\\
	\displaystyle
	\int_\Omega y_t(0)\varphi(0)dx+\int_\Omega u_t(0)\psi(0)dx
	-\int_\Omega\varphi_t(0)y(0)dx\\
	\displaystyle - \int_\Omega\psi_t(0)u(0)dx
	-\int_\Omega bu(0)\varphi(0)dx +\int_\Omega by(0)\psi(0)dx  + \int_{0}^{T} \int_\Omega c(x)v(t)\psi dx dt .\label{s123.6}
	\end{array}
	\right.
	\end{equation}
	Note that $\HH'=(H^{-1}(\Omega)\times L^2(\Omega))^2$. Then we have 
	\begin{equation}\left\lbrace
	\begin{array}{ll}
	\displaystyle \langle  \,(u_t(T,x),-u(T,x),y_t(T,x),-y(T,x)) , \Phi(T)  \,\rangle_{\HH'\times
		\HH} =\\
	\langle  (u_1,-u_0,y_1,-y_0) , \Phi_0 \rangle_{\HH'\times\HH}
	+\displaystyle \int_0^T\int_\Omega cv(t)\psi dxdt =L(\Phi_0).\label{0ecl2.6}	\end{array}
	\right.
	\end{equation}

	Using the direct observability inequality \eqref{IO}, we deduce that 
	\begin{equation}
	\displaystyle \prl L\prl_{ \LL(\HH,\rr)} \, \,\leq \hskip 0.1 cm
	\prl v_0\prl_{L^2(0,T;L^2(\omega_{c_+}))}+\prl U_0\prl_{\HH'}.
	\end{equation}
	Using the Riesz representation theorem, there exists an element 
	$\ZZ(x,t) \, \in\, \HH'$ solution of 
	\begin{equation}
	L(\Phi_0)= \,\langle \,\ZZ,\Phi_0 \,\rangle_{\HH' \times \HH}, \
	\ \forall \Phi_0\in {\HH}.
	\end{equation}
	Then, define the  weak solution $U(x,t)$ of system \eqref{e11.2c}-\eqref{e13.2c} by  $U(x,t)= \ZZ(x,t)$.
The proof is thus complete.	
{$\quad\square$}\end{proof}

Next, we consider the indirect locally internal exact controllability problem: For given $T > 0$ (sufficiently large) and initial data $U_0$, does there exists a suitable control $v$ that brings back the solution to equilibrium at time $T$, that is such the solution of \eqref{e11.2c}-\eqref{e13.2c} satisfies $ u(T)=u_t(T) =y(T)=y_t(T)=0$. Indeed, applying the HUM method, we obtain the following result. 
\begin{theorem}
Let $a=1$. Assume that conditions {\rm{(LH1)}} and {\rm{(LH2)} }hold. Assume also that $\omega_{b} \subset \omega_{c_+}$  satisfies the geometric control condition $\rm{GCC}$ and that $b$, $c \in W^{1,\infty} (\Omega)$.
	For every $ T > M_1$, where $ M_1$ is given in Theorem \ref{IO.2} and for every 
	$$U_{0} \in (L^{2}(\Omega) \times \\H^{-1}(\Omega))^{2},$$
	there exists a control 
	$$v(t) \in [H^1(0,T;L^2(\omega_{c_+}))]' ,$$
	such that the solution of the  controlled system \eqref{e11.2c}-\eqref{e13.2c} satisfies\\
	$$ u(T) =u_t(T)=y(T)=y_t(T)=0.$$
\end{theorem}
\begin{proof}
	We will apply the HUM method. Thanks to the indirect observability inequalities \eqref{IO}, we consider the seminorm defined by
	$$ \norm{\Phi_{0}}^2_{\HH}= \int_0^T\int_{\omega_{b}}|\psi_t|^2dxdt,$$
	where $\Phi=( \psi, \psi_{t},\varphi, \varphi_{t})$ designate the solution of the homogeneous problem \eqref{eh11.2}-\eqref{eh13.2}.\\
	Take the control $v=\frac{d}{dt} \psi_{t}$.
	Now, we solve the following time reverse problem:
	\begin{equation}
	\left\{
	\begin{matrix}
	\zeta_{tt}- \Delta \zeta +b \chi_{t} & = &  c \frac{d}{dt}\psi _{t} & {\rm in}&  (0,T) \times \Omega,\\
	\chi_{tt}- \Delta \chi - b \zeta_{t} &=&  0 & {\rm in} & (0,T)\times \Omega,\\
	\chi(T)=\chi_t(T)=\zeta(T)=\zeta_t(T)&= &0. \label{Retrograde}
	\end{matrix}
	\right.
	\end{equation}
	By Theorem \ref{t2}, the system \eqref{Retrograde} admits a solution 
	$$ \Psi(x,t)= ( \zeta, \zeta_{t},\chi,\chi_{t},) \in C^0([0,T], H').$$
	We define the linear operator $\Lambda$ by:
	$$\Lambda :\HH = (H_{0}^1(\Omega) \times L^{2}(\Omega))^2 \rightarrow (H^{-1}(\Omega)\times L^2(\Omega) )^{2},$$
	where
	$$\Lambda \Phi_0 = (\zeta_t(0),-\zeta(0),\chi_t(0),-\chi(0)) \, \hskip 1cm \, \forall \,\, \Phi_{0} \in (H_0^1(\Omega)\times L^2(\Omega)).$$
	In addition, we define the following linear form
	\begin{equation}
	\langle\Lambda \Phi_0,\tilde{\Phi_0} \rangle=\int_0^T \int_{\omega_c} \psi_t \tilde{\psi}_t dxdt =(\Phi _0, \tilde{\Phi}_0) _\HH, \ \ \quad\forall \, \, \tilde{\Phi}_0\in\HH,
	\label{Crochet}
	\end{equation}
	where $(.,.)_\HH$ is the scalar product associated to the norm $\norm{.}_\HH$.
	
	Using Cauchy-Schwarz in \eqref{Crochet} , we deduce that
	\begin{equation}
	|\langle\Lambda \Phi _{0}, \tilde{\Phi_{0}} \rangle _{\HH \times \HH'}| \leqslant \norm{\Phi_{0} }_\HH  \norm{\tilde{\Phi_0}}_\HH, \quad   \forall \,  \Phi_{0}, \tilde{\Phi}_{0}\in \HH.
	\label{1e3595}
	\end{equation}
	In particular, we have
	$$ |\langle\Lambda \Phi _{0},{\Phi_{0}} \rangle _{\HH \times \HH'}| = \norm{\Phi_{0} }_\HH ^2 \quad \quad \forall \, \Phi_0 \in \HH.  $$
	
	Then the inverse inequality in Theorem \ref{IO.2} implies that the operator $\Lambda$  is coercive and continuous over $\HH$. Thanks to the Lax-Milgram theorem, we have $\Lambda $ is an isomorphism from $\HH$ into $\HH'$ .	In particular, for every $U_0 \in (L^2(\Omega)\times H^{-1}(\Omega))^2$, there exists a solution $\Phi_0 \in \HH$, such that\\
	$$\Lambda (\Phi_0)=-U_0=(\zeta_t(0),-\zeta(0),\chi_t(0),-\chi(0)). $$	
It follows from the uniqueness of  the solution of problem \eqref{Retrograde} that
	$$U=\Psi.$$
	Consequently, we have 
	$$ u(T)=u_{t}(T)=y(T)=y_{t}(T)=0.$$
	The proof is thus complete.
{$\quad\square$}\end{proof}
\section{ Exponential stability and exact controllability in the case $a \neq 1$}
\label{section4}
\subsection{Exponential stability in the weak energy space}
The aim of this subsection is to show the exponential stability of system \eqref{e11.2}-\eqref{e13.2} in a weak energy space in the case when the waves do not propagate with same speed, i.e., $a\neq1$. For this sake, we define the weak energy space
$$D=H_0^1(\Omega) \times L^2(\Omega) \times L^2(\Omega) \times H^{-1}(\Omega)$$
equipped with the scalar product : for all $U=(u,v,y,z) \in D$ and $\tilde U=(\tilde u,\tilde v,\tilde y, \tilde z) \in D $,
$$ (U, \tilde U)=\int_{\Omega}(a\nabla u. \nabla \tilde u + v \tilde v+ y\tilde y+(-\Delta)^{-1/2}z(-\Delta)^{-1/2}\tilde z)dx.$$

Next, we define the unbounded linear operator $\AA_d:D(\AA_d) \subset D \rightarrow D$ by
$$\AA_dU
=(\,v , a \Delta u - b z- c v, \,z,\Delta y+b v\,),
$$
$$D(\AA_d)=\big((H_0^1(\Omega)\cap H^2(\Omega)) \times H_0^1(\Omega) \times H_0^1(\Omega) \times L^2(\Omega) \big), \hskip 1 cm \forall\, \,  U\,=\,(u ,v, y , z)\,
\in\,D(\AA_d).
$$
We define the partial energy associated to a solution $U=(u ,u_t, y , y_t)$ of \eqref{e11.2}-\eqref{e13.2} by 
$$e_1(t)=\frac{1}{2}\big(a \norm{\nabla u}^2_{L^2(\Omega)}+ \norm{u_t}^2_{L^2(\Omega)}\big). $$
We define also the weakened partial energy by
$$\tilde e_2(t)= \frac{1}{2}\big(\norm{y_t}_{H^{-1}(\Omega)}^{2}+ \norm{y}^2_{L^2(\Omega)}\big)$$
and the total mixed energy by 
$$E_m(t)=e_1(t)+ \tilde e_2(t).$$

In order to study the exponential decay rate, we need to assume that $\omega_{c_+}$ satisfies the geometric conditions PMGC. Then there exist $\varepsilon>0$, subsets $\Omega_j\subset \Omega$, $j=1,...,J$, with Lipschitz boundary $\Gamma_j=\partial \Omega_j$ and points $x_j\in \rr^N$ such that $\Omega_i\cap\Omega_j =\emptyset$ if $i\not= j$ and $\omega_c^{+} \supset \NN_\epsilon \left( \displaystyle\cup _{j=1}^{J} \gamma_j\left(x_j\right) \cup \left( \Omega \setminus \displaystyle \cup _{j=1}^{J} \Omega_j \right)\right) \cap \Omega$ with $\NN_\epsilon(\OO)  = \{x \in \mathbb{R}^N: d(x,\OO)< \varepsilon \}$ where $\OO\subset \rr^N$, $\gamma_j(x_j) = \{ x \in  \Gamma_j:(x-x_j)\cdot \nu_j(x) > 0 \}$ where $\nu_j$ is the outward unit normal vector to $ \Gamma_j$ and that
 $\omega_b$  satisfies the GCC condition and  $$\omega_{b}\subset \left( \Omega \setminus \displaystyle \cup _{j=1}^{J} \Omega_j \right). \quad \quad \quad \rm(LH3)$$
Now, we are ready to establish the following main theorem of this section:
\begin{theorem}\label{ThexpE} (Exponential decay rate) 
Let $a\neq1$. Assume that conditions {\rm (LH1)} and {\rm (LH2)} hold. Assume also that $\omega_{c_+}$  satisfies the geometric conditions  {\rm PMGC}, $\omega_b$ satisfies {\rm{GCC}} condition and {\rm (LH3)} and $b,\, c \in L^\infty(\Omega)$. Then there exist positive constants $M\geq 1$, $\theta >0$
	such that for all initial data $(u_0,u_1,y_0,y_1)\in D$ the energy
	of system (\ref{e11.2})-(\ref{e13.2}) satisfies the following decay
	rate:
	\begin{equation}
	E_m(t)\leq Me^{-\theta t}E_m(0),\ \ \ \ \forall t>0.\label{ExpSta.2}
	\end{equation}
\end{theorem}
{In order to prove the above theorem, we apply the same strategy 
 using  Huang \cite{Huang-85} and Pr\"uss \cite{Pruss-84}.}
A $C_0$- semigroup of contraction $(e^{t\AA})_{t\geqslant 0}$ in a Hilbert space $\HH$ is exponentially stable if and only if
\begin{align}
i \rr \subseteq \rho(\AA_d)   \quad \mbox{ and } &\label{H1}\tag{H1}  \\
\limsup_{\beta \, \in \rr, |\beta| \rightarrow + \infty}\parallel (i \beta I - \AA_d )^{-1} \parallel_{\mathcal{L}(D)} <\infty&\label{H2}\tag{H2} 
\end{align}

Condition \eqref{H1} was already proved. We now prove that condition \eqref{H2} holds, using an argument of contradiction. For this aim, we suppose that there exist a  real sequence $\beta_n $ with $\beta_n \rightarrow +\infty$ and a sequence $U_n=(u_n,v_n,y_n,z_n) \in D(\AA_d)$ such that 	
\begin{align}
\parallel U_n \parallel_{D} &= 1 \quad \mbox{ and } \label{E1.2} \\
\lim_{n\rightarrow\infty}\parallel  (i \beta_n I -
\AA_d) U_n\parallel_{D} &= 0 \ .\label{E2.2}
\end{align}
Next, detailing equation (\ref{E2.2}), we get
\begin{eqnarray}
i\beta_n u_n-v_n &=& f_n^1 \,\,\rightarrow \,0 \hskip 0.5 cm \mbox{in}\hskip 0.5 cm H_0^1(\Omega),\label{e46.2}\\
i\beta_n v_n- a\Delta u_n +b z_n + cv_n&=& g_n^1 \,\,\rightarrow \,0 \hskip 0.5 cm \mbox{in}\hskip 0.5 cm L^2(\Omega),\label{e47.2}\\
i\beta_n y_n-z_n &=& f_n^2 \,\,\rightarrow \,0 \hskip 0.5 cm \mbox{in}\hskip 0.5 cm L^2(\Omega),\label{e48.2}\\
i\beta_n z_n- \Delta y_n - b v_n  &=& g_n^2 \,\,\rightarrow
\,0 \hskip 0.5 cm \mbox{in}\hskip 0.5 cm H^{-1}(\Omega).\label{e49.2}
\end{eqnarray}
Eliminating $v_n$ and $z_n$ from the previous system, we obtain the following system
\begin{equation}
\beta_n^2 u_n +a \Delta
u_n-i\beta_n b y_n-i\beta_n c u_n=-g_n^1-bf_n^2
-i\beta_nf_n^1-cf_n^1  \hskip 0.5 cm \mbox{in}\hskip 0.5 cm L^2(\Omega), \label{E5.2}
\end{equation}
\begin{equation}
\beta_n^2 y_n + \Delta y_n +i \beta_n b u_n = -i \beta_n
f_n^2 + b f_n^1 -g_n^2 \hskip 0.5 cm \mbox{in}\hskip 0.5 cm H^{-1}(\Omega). \label{E4.2}
\end{equation}
From \eqref{E1.2}, we have $\nabla u_n$, $v_n$ and $y_n$ are uniformly bounded in $L^2(\Omega)$ and $ z_n$ is uniformly bounded in $H^{-1}(\Omega)$. 
 Using now \eqref{E1.2} and \eqref{e46.2}, we deduce that $\beta_n u_n$ is uniformly bounded in $L^2(\Omega)$. In addition, using \eqref{E1.2} and \eqref{e48.2}, we deduce that $\beta_n y_n$  is uniformly bounded in $H^{-1}(\Omega)$. More precisely,  
 $$\norm{u_n}_{L^2(\Omega)}=\frac {O(1)}{\beta_n}=o(1) \quad {\rm{and}} \quad  \norm{y_n}_{H^{-1}(\Omega)}=\frac {O(1)}{\beta_n}=o(1).$$ 
\begin{lemma}
The solution $(u_n,v_n,y_n,z_n) \in D(\AA_d)$ of system (\ref{e46.2})-(\ref{e49.2}) satisfies the following estimates
\begin{equation}
	\displaystyle \int_\Omega c |\beta_n u_n|^2dx=o(1) \quad and \quad \int_{\omega_{c_+}} |\beta_nu_n|^2 dx = o(1).\label{E33.2}
\end{equation}
\end{lemma}
\begin{proof}
	First, since $U_n$ is uniformly bounded in $D$ and using \eqref{E2.2}, we get 
	\begin{equation}
	\mathrm{Re}\left\{i\beta_n \parallel U_n \parallel^2 - (\AA_d U_n,
	U_n)\right\}= \int_{\Omega}c(x)|v_n|^2 dx=o(1).\label{E3.2}
	\end{equation}

	Next, using equations \eqref{E3.2} and \eqref{e46.2},  we get
	\begin{equation}
	\int_{\Omega}c|\beta_n u_n|^2 dx =o(1).
	\end{equation}
Under condition (LH1), it follows
	\begin{equation*}
	\int_{\omega_{c_+}}|\beta_n u_n|^2 dx =o(1). 
	\end{equation*}
	The proof is thus complete.
{$\quad\square$}\end{proof}
Now as $\omega_{c^+}$ satisfies the PMGC condition, let the reals $0 < \varepsilon_1 < \varepsilon_2 <\varepsilon$ and define   
$$Q_ {i}= \NN_{\varepsilon_i}  \left( \displaystyle\cup _{j=1}^{J} \gamma_j\left(x_j\right) \cup \left( \Omega \setminus \displaystyle \cup _{j=1}^{J} \Omega_j \right)\right), \quad i=1,2.$$
Since $\overline { \mathbb{R}^N\setminus \omega_{c^+}} \cap\overline Q_2 =\emptyset$, we can construct a function $\hat{\eta} \in C_0^\infty(\Omega)$ such that
\begin{equation*}
\hat{\eta}(x)=0 \hskip 0.5cm \mbox{if}\,\, \,x \in \Omega\setminus \omega_{c_+}, \quad
0\leq \hat \eta(x) \leq 1, \quad
\hat{\eta}(x)=1\hskip 0.5cm \mbox{if} \, \,\,\,x \in  Q_{2}.
\label{FT}
\end{equation*}
\begin{lemma}
The solution $(u_n,v_n,y_n,z_n) \in D(\AA_d)$ of system (\ref{e46.2})-(\ref{e49.2}) satisfies the following estimates
\begin{equation}
\int_{\Omega} \hat{\eta} \mid\nabla u_n\mid^2dx=o(1) \hskip 0.2 cm \mbox{and} \hskip 0.2 cm
\int_{Q_2 \cap \Omega} \mid\nabla u_n\mid^2 dx
=o(1). \label{E10.2}
\end{equation} 	
\end{lemma}

\begin{proof}
First,	multiplying equation (\ref{E5.2}) by $\hat\eta \bar{u}_n$. Then, using Green's formula  and the fact that $u_n=0 $ on $\Gamma$, we obtain 
	\begin{equation}\left\{
	\begin{array}{l}
	\displaystyle \int_{\Omega}\hat\eta|\beta_n u_n|^2dx-a\int_{\Omega}\hat\eta\mid
	\nabla u_n\mid^2dx-a
	\int_{\Omega}\overline{u}_n
	(\nabla\hat \eta\cdot \nabla u_n)dx
	\displaystyle - \,i \,\beta_n\int_{\Omega}  b \hat\eta y_n\overline u_n dx  \\ \\
	\displaystyle -i\beta_n\int_{\Omega} c\hat\eta |u_n|^2 dx=
	\displaystyle\,  \,
	\int_{\Omega}(-g_n^1-bf_n^2 -i\beta_nf_n^1-cf_n^1)\hat\eta \overline{u}_ndx.
	\label{e415.2}
	\end{array}
	\right.
	\end{equation}
	As $f_n^1$  converges to zero in $H_0^1(\Omega)$,  $f_n^2$, $g_n^1$ converge to zero in $L^2(\Omega)$ and $\beta_n u_n$ is uniformly bounded in $L^2(\Omega)$, we get
	\begin{equation}
	\int_{\Omega}(-g_n^1-bf_n^2 -i\beta_nf_n^1-cf_n^1)\hat\eta \overline{u}_ndx=o(1). \label{s18.Iyz.2}	
	\end{equation}
	Using the fact that $\nabla u_n$, $y_n$ are uniformly bounded in $L^2(\Omega)$, $\norm{ u_n} _{L^2(\Omega)} =o(1)$ and estimation \eqref{E33.2}, we will have
	\begin{equation}
\displaystyle \int_{\Omega}\hat\eta|\beta_n u_n|^2dx-a
\int_{\Omega}\overline{u}_n
(\nabla\hat \eta\cdot \nabla u_n)dx
\displaystyle - \,i \,\beta_n\int_{\Omega}  b \hat\eta y_n\overline u_n dx  \\ \\
\displaystyle -i\beta_n\int_{\Omega} c\hat\eta |u_n|^2 dx=	o(1). \label{s18.Iy.2}
\end{equation}	
Finally, inserting 	\eqref{s18.Iyz.2} and  	\eqref{s18.Iy.2} into \eqref{e415.2}, we deduce
	\begin{equation*}
	\int_{\Omega} \hat\eta \mid\nabla u_n\mid^2dx=o(1) \hskip 0.2 cm \mbox{and} \hskip 0.2 cm
	\int_{Q_2 \cap \Omega} \mid\nabla u_n\mid^2 dx
	=o(1). 
	\end{equation*} 
	The proof is thus complete.
{$\quad\square$}\end{proof}
\begin{lemma}
The solution $(u_n,v_n,y_n,z_n) \in D(\AA_d)$ of system (\ref{e46.2})-(\ref{e49.2}) satisfies the following estimate
\begin{equation}
\int_{\omega_{b}}|y_n|^2dx=o(1). \label{y.2}
\end{equation}
\end{lemma}
\begin{proof}
The proof contains two steps.\\
\textbf{Step 1. (Boundedness of $\frac{1}{\beta_n} \nabla y_n$).} 
Multiplying equation \eqref{E4.2} by $\frac{1}{\beta_n^2}\overline y_n$, we obtain
\begin{align}\label{y2.2}
\int_{\Omega} |y_n|^2dx + <\Delta y_n, \frac{1}{\beta_n^2}\overline y_n>_{H^{-1} (\Omega) \times H_0^1 (\Omega)}=&-i\int_{\Omega}\frac{1}{\beta_n} f_n^2\overline y_ndx  + \int_{\Omega}b f_n^1 \frac{1}{\beta_n^2}\overline y_ndx \\
&- <g_n^2,\frac{1}{\beta_n^2}\overline y_n>_{H^{-1}(\Omega) \times H_0^1 (\Omega)}. \nonumber 
\end{align}  
Since $f_n^1$ converges to zero in $H_0^1(\Omega)$, $f_n^2$ converges to zero in $L^2(\Omega)$ and $y_n$ is uniformly bounded in $L^2(\Omega)$, we get
\begin{equation}
-i\int_{\Omega}\frac{1}{\beta_n} f_n^2\overline y_ndx   +\int_{\Omega}b f_n^1 \frac{1}{\beta_n^2}\overline y_ndx=o(1). \label{y3.2}
\end{equation}  
Inserting \eqref{y3.2} into \eqref{y2.2}, we will have
after integrating by parts 
\begin{equation*} 
\int_{\Omega}\left|\frac{\nabla y_n}{\beta_n}\right|^2dx =  \int_{\Omega} |y_n|^2dx  +<g_n^2,\frac{1}{\beta_n^2}\overline y_n>_{H^{-1} (\Omega) \times H_0^1 (\Omega)}+o(1). 
\end{equation*} 
Using Cauchy-Schwarz and Young inequalities in the previous equation, we obtain that 
\begin{equation*}
\frac{1}{2}\left\|\frac{\nabla y_n}{\beta_n}\right\|^2_{L^2(\Omega)} \leq \norm{y_n}^2_{L^2(\Omega)} + \frac{1}{2}\norm{g_n^2}^2_{H^{-1}(\Omega)} +o(1).
\end{equation*}	 
It follows, from the uniform boundedness of $y_n$ in $L^2(\Omega)$ and $g_n^2$  in $H^{-1}(\Omega)$, that
\begin{equation} 
\left\|\frac{\nabla y_n}{\beta_n}\right\|^2=O(1).\label{y4.2}
\end{equation}
\textbf{Step 2. (Main asymptotic estimation).}
Multiplying equation \eqref{E5.2} by $\hat{\eta} \frac{1}{\beta_n} \overline y_n$. Later, using Green's formula and the fact that $y_n=0$ on $\Gamma$, we get
	\begin{equation}\left\{
\begin{array}{l}
\displaystyle\int_{\Omega} \hat\eta \beta_n u_n \overline y_ndx -a \int_{\Omega}\frac{1}{\beta_n}\hat\eta(\nabla u_n \cdot \nabla \overline y_n)dx - a \int_{\Omega} \frac{1}{\beta_n}(\nabla \hat\eta\cdot\nabla u_n) \overline y_ndx \\ \\
\displaystyle
-i\int_{\Omega} b \hat \eta |y_n|^2dx -i\int_{\Omega}c u_n \hat \eta \overline y_ndx= \int_{\Omega}(-g_n^1-bf_n^2
-i\beta_nf_n^1-cf_n^1)\frac{\hat \eta}{\beta_n}\overline y_ndx.\label{y1.2}
\end{array}
	\right.
\end{equation}
 Next, using the definition of $\hat \eta$ and equations \eqref{E10.2} and \eqref{y4.2}, we get
\begin{equation} \label{y5.2} 
 -a \int_{\Omega}\frac{1}{\beta_n}\hat\eta(\nabla u_n. \nabla \overline y_n)dx =o(1).
 \end{equation}
  Using \eqref{E33.2}, \eqref{E10.2} and the fact that $y_n$ is uniformly bounded in $L^2(\Omega)$, we obtain
  \begin{equation}\label{y6.2}
  -a \int_{\Omega} \frac{1}{\beta_n}(\nabla \hat\eta.\nabla u_n) \overline y_ndx- \int_{\Omega} \hat\eta \beta_n u_n \overline y_ndx - i\int_{\Omega}c u_n \hat \eta y_ndx=o(1).
 \end{equation} 
 Using the fact that $f_n^1$ converges to zero in $H_0^1(\Omega)$, $f_n^2$, $g_n^1$ converge to zero in $L^2(\Omega)$ and $y_n$ is uniformly bounded in $L^2(\Omega)$, we will have 
\begin{equation} \label{y9.2}
 \int_{\Omega}(-g_n^1-bf_n^2
-i\beta_nf_n^1-cf_n^1)\frac{\hat \eta}{\beta_n}\overline y_ndx =o(1).
\end{equation}
Finally, inserting  \eqref{y5.2}-\eqref{y9.2}
into \eqref{y1.2}, we get
$$\int_{\Omega} b \hat \eta |y_n|^2dx = o(1).$$ 
It follows, from condition (LH3), that 
$$\int_{\omega_{b}}|y_n|^2dx=o(1).$$
The proof is thus complete.
{$\quad\square$}\end{proof}
\begin{lemma}
The solution $(u_n,v_n,y_n,z_n) \in D(\AA_d)$ of system (\ref{e46.2})-(\ref{e49.2}) satisfies the following estimate
\begin{equation}\label{ally.2}
\int_{\Omega}|y_n|^2dx=o(1).
\end{equation}
\end{lemma}
\begin{proof}
Noting that $\omega_b$ satisfies the GCC condition, so we can 
taking $f_n= y_n$ in Lemma \ref{auxiliary.2}.
 Multiplying equation \eqref{E4.2} by $ \overline\phi_n$. Then, we have
\begin{align}\label{aux.23}
\int_{\Omega}\beta_n^2 \overline \phi_n y_ndx- <\Delta y_n,\overline \phi_n>_{H^{-1}(\Omega) \times H_0^1(\Omega)}+i \int_{\Omega}\beta_n b u_n \overline \phi_ndx=& -i\int_{\Omega} \beta_nf_n^2 \overline \phi_n dx  \nonumber \\ 
&+ \int_{\Omega} b f_n^1\overline \phi_ndx \\
&-<g_n^2, \overline \phi_n>_{H^{-1}(\Omega) \times H_0^1(\Omega)} \nonumber .
\end{align}
Using the fact that $\phi_n \in H^2(\Omega) \cap H_0^1 (\Omega)$ and $y_n \in H_0^1(\Omega)$, then we have
\begin{equation}
- <\Delta y_n,\overline \phi_n>_{H^{-1}(\Omega) \times H_0^1(\Omega)}=\int_{\Omega} y_n\Delta \overline \phi_n dx.
\end{equation}
It follows, from the first equation of \eqref{auxpro12.2} and \eqref{aux.23}, that 
\begin{align}\label{yaux.2}
&\int_{\Omega}|y_n|^2dx=i  \int_{\Omega} b \beta_n \overline \phi_n y_ndx -i \int_{\Omega}\beta_n b u_n \overline \phi_ndx\\ 
& -i\int_{\Omega} \beta_nf_n^2 \overline \phi_n dx+ \int_{\Omega} b f_n^1\overline \phi_ndx -<g_n^2, \overline \phi_n>_{H^{-1}(\Omega) \times H_0^1(\Omega)}. \nonumber
\end{align}
Using the fact that $\beta_n \phi_n$ is uniformly bounded in $L^2(\Omega)$, $f_n^1$ converges to zero in $H_0^1(\Omega)$, $f_n^2$ converges to zero in $L^2(\Omega)$ , $g_n^2$ converges to zero in $H^{-1}$, \eqref{y.2} and $\norm{u_n}=o(1)$ in equation \eqref{yaux.2}, we obtain
\begin{align}\label{yaux1.2}
\int_{\Omega}|y_n|^2dx=o(1). 
\end{align}
The proof is thus complete.
{$\quad\square$}\end{proof}

\begin{lemma}
The solution $(u_n,v_n,y_n,z_n) \in D(\AA_d)$ of system (\ref{e46.2})-(\ref{e49.2}) satisfies the following estimate
\begin{equation}\label{z.2}
\int_{\Omega}|\beta_n(-\Delta)^{-1/2}y_n|^2dx=o(1).
\end{equation}
\end{lemma}
\begin{proof}
Multiplying equation \eqref{E4.2} by $(-\Delta)^{-1}\overline y_n$, then integrating by parts and using the fact that $y_n=0$ on $\Gamma$, we get
\begin{align}\label{betafgy.2}
\int_{\Omega}|\beta_n(-\Delta)^{-1/2}y_n|^2dx&= \int_{\Omega }|y_n|^2dx -i \int_{\Omega} \beta_n b u_n (-\Delta)^{-1}\overline y_ndx \nonumber\\
&-i\int_{\Omega}\beta_n(-\Delta)^{-1/2} f_n^2 (-\Delta)^{-1/2}\overline y_ndx
 +\int_{\Omega}b f_n^1 (-\Delta)^{-1}\overline y_ndx\\
 &-<g_n^2,(-\Delta)^{-1}\overline y_n>_{H^{-1}(\Omega) \times H_0^1 (\Omega)}.
 \nonumber
\end{align}	
Using Cauchy-Schwarz and Poincar\'e inequalities, we get
\begin{align}\label{betaf2y.2}
\displaystyle\abs{\int_{\Omega}\beta_n(-\Delta)^{-1/2} f_n^2 (-\Delta)^{-1/2}\overline y_ndx }&\leq \norm{(-\Delta)^{-1/2} f_n^2}_{L^2(\Omega)}\norm{\beta_n(-\Delta)^{-1/2}\overline y_n}_{L^2(\Omega)}\\
&\leq c_0 \norm{ f_n^2}_{L^{2}(\Omega)}\norm{\beta_n\overline y_n}_{H^{-1}(\Omega)}.\nonumber
\end{align}

 It follows, from the convergence to zero of $f_n^2$ in $L^2(\Omega)$ and  the boundedness of $\beta_ny_n$  in $H^{-1}(\Omega)$, that
\begin{equation}\label{betayy.2} 
\int_{\Omega}\beta_n(-\Delta)^{-1/2} f_n^2 (-\Delta)^{-1/2}\overline y_ndx=o(1).
\end{equation}
Similarly, we have
\begin{align}
\abs{<g_n^2,(-\Delta)^{-1}\overline y_n>_{H^{-1}(\Omega) \times H_0^1 (\Omega)}}&= \int_{\Omega}(-\Delta)^{-1/2}g_n^2 (-\Delta)^{-1/2}\overline y_ndx\\
&\leq \norm{(-\Delta)^{-1/2}g_n^2 }_{L^2(\Omega)} \norm{(-\Delta)^{-1/2}\overline y_n }_{L^2(\Omega)} \nonumber\\
&\leq \norm{g_n^2 }_{H^{-1}(\Omega)} \norm{\overline y_n}_{H^{-1}(\Omega)}. \nonumber
\end{align}
It follows, from the convergence of $g_n^2$ and $y_n$ to zero in $H^{-1}(\Omega)$, that 
\begin{equation}\label{betay}
<g_n^2,(-\Delta)^{-1}\overline y_n>_{H^{-1}(\Omega) \times H_0^1 (\Omega)}=o(1).
\end{equation}

Note that $(-\Delta)^{-1}$ is compact operator from $L^2$ to $L^2$. Then $(-\Delta)^{-1}y_n$ is uniformly bounded in $L^2$.
Finally, using \eqref{E33.2}, \eqref{ally.2}, \eqref{betayy.2}, \eqref{betay} and the fact that $f_n^1$ converges to zero in $H_0^1(\Omega)$ into equation \eqref{betafgy.2},  we deduce 
 \begin{equation*}
 \int_{\Omega}|\beta_n(-\Delta)^{-1/2}y_n|^2dx =o(1).
 \end{equation*}
 The proof is thus complete. 
{$\quad\square$}\end{proof}
\begin{lemma}
	The solution $(u_n,v_n,y_n,z_n) \in D(\AA_d)$ of system (\ref{e46.2})-(\ref{e49.2}) satisfies the following estimate
	\begin{equation} \label{pww.2}
		\displaystyle \int_{\Omega \setminus (Q_2\cap \Omega)} \big(|\nabla u_n|^2 + |\beta_n u_n|^2) dx
		=o(1).
	\end{equation} 
\end{lemma}
\begin{proof}
	Since $(\overline{\Omega}_{j}\setminus Q_{2}) \cap \overline {Q_{1}} = \emptyset$, we define the function $\psi_j \in C^\infty_0(\rr^N)$ by:
	\begin{equation*}
		\psi_j (x)=0 \mbox{ \; if } x \in Q_{{1}}, \,\,\,\,\, 0 \leqslant \psi_j \leqslant 1, \,\,\,\,\, \psi_j (x)=1 \mbox{ \; if } x \in \overline{ \Omega_j}\setminus Q_{2}.
	\end{equation*}
	For $m_j(x) = (x-x_j)$, we define $h_j(x) = \psi_j(x)m_j(x)$. 
	
	 Multiplying equation (\ref{E5.2}) by  $2(h_j\cdot\nabla \overline{u}_n)$ and integrating over $\Omega_j$, using the dissipation \eqref{E33.2} and the fact that $\nabla u_n$ is uniformly bounded in $L^2(\Omega)$, we obtain 
	\begin{eqnarray}
		\begin{array}{l}\label{pw15.2}
			\displaystyle 2\beta_n^2\int_{\Omega_j} u_n (h_j\cdot\nabla \overline u_n)dx+   2a  \int_{\Omega_j}\Delta u_n(h_j\cdot \nabla \overline u_n)dx
			\displaystyle -2i\int_{\Omega_j}\beta_n b y_n(h_j\cdot \nabla \overline u_n)dx=\\
			\displaystyle 2\int_{\Omega_j}(-g_n^1-bf_n^2
		 -cf_n^1)(h_j \cdot \nabla\overline u_n)dx
		-2i\int_{\Omega_j}
			\beta_nf_n^1(h_j \cdot \nabla\overline u_n)dx.
		\end{array}
	\end{eqnarray}
i) \textbf{Estimation of the second member of \eqref{pw15.2}}.
First, using Green's formula and the fact that $u_n =0$ on $(\Gamma_j \setminus \gamma_j) \cap \Gamma$ and $h_j =0$ on $\gamma_j$, we get
	\begin{eqnarray}
		-2i\int_{\Omega_j} \beta_n f_n^1(h_j \cdot\nabla \overline u_n)
		dx= 2i\int_{\Omega_j}\beta_n \overline u_n(h_j \cdot \nabla f_n^1)dx+ 2i\int_{\Omega_j}\beta_n \overline u_n f_n^1 ({\rm{div}} h_j)dx.
	\end{eqnarray}
	So, from the fact that $f_n^1$ converges to zero in $H_0^1(\Omega)$ 
	and $\beta_nu_n$ is uniformly bounded in $L^2(\Omega)$, we obtain
	\begin{equation}\label{pwl.2}
	-2i\int_{\Omega_j} \beta_n f_n^1(h_j \cdot\nabla \overline u_n)
dx= o(1).	
	\end{equation}
Next, as $f_n^1$ converges to zero in $H_0^1(\Omega)$, $f_n^2$, $g_n^1$ converge to zero in $L^2(\Omega)$ and the sequence $(\nabla u_n)$ is uniformly bounded in $L^2(\Omega)$,  we deduce
	\begin{equation}
		2\int_{\Omega_j}(-g_n^1-bf_n^2
	-cf_n^1)(h_j. \nabla\overline u_n)dx=o(1).\label{pwa.2}
	\end{equation}
Finally, we deduce that the second member of \eqref{pw15.2} is $o(1)$.

ii) \textbf{Estimation of the first integral of equation \eqref{pw15.2}.} 
Using Green's formula, we get
\begin{equation}\label{pw55.2}
\mbox{Re} \bigg\{ 2 \int_{\Omega_j}\beta _n^2 u_n (h_j \cdot \nabla \overline u_n)dx \bigg\}= - \int_{\Omega_j} ({\rm{div}}{h_j}) |\beta_n u_n|^2dx + \int_{\Gamma_j} (h_j \cdot \nu_j)|\beta_n u_n|^2d\Gamma_j.  
\end{equation}
Since $\Psi_j=0$ on $\gamma_j$ and $u_n=0$ on $(\Gamma_j \setminus \gamma_j) \cap \Gamma$, then we have
\begin{equation}\label{pw5.2}
\mbox{Re} \{ 2 \int_{\Omega_j}\beta _n^2 u_n (h_j \cdot \nabla \overline u_n)dx \}= - \int_{\Omega_j} ({\rm{div}}{h_j}) |\beta_n u_n|^2dx.
\end{equation}
iii) \textbf{Estimation of the second integral of equation \eqref{pw15.2}.} Using Green's formula, we get 
	\begin{align} \label{pw4.2} 
		\mbox{Re} \bigg\{2a\int_{\Omega_j}\Delta u_n(h_j\cdot\nabla \overline{u}_n)\bigg\} =
		-2a	\mbox{Re}\bigg\{\displaystyle{\sum_{i,k=1}^{N}}\int_{\Omega_j}\partial_ih_j^k\partial_i u_n \partial_k u_ndx\bigg\} + \\
		a \int_{\Omega_j}({\rm{div}} h_j) | \nabla u_n | ^2 dx 
		 -a\int_{\Gamma_j}(h_j \cdot \nu_j) |\nabla  u_n|^2d\Gamma_j + 2a 	\mbox{Re}\bigg\{\int _{\Gamma_j} \partial_{\nu_j} u_n(h_j \cdot \nabla \overline u_n)d\Gamma_j\bigg\}.\nonumber 
  	\end{align}
According to the choice of $\psi_j$, only the boundary terms over $(\Gamma_j \setminus \gamma_j) \cap \Gamma$ are non vanishing in (\ref{pw4.2}). But on this part of the boundary   $ u_n = 0$, and consequently $ \nabla u_n = (\partial_{\nu} u_n)\cdot  \nu = (\partial_{\nu_j} u_n) \nu_j.$ 
Then, we have 
\begin{equation}\label{bounterms.2}
-a\int_{\Gamma_j}(h_j \cdot \nu_j) |\nabla  u_n|^2d\Gamma_j + 2a\mbox{Re} \bigg\{\int _{\Gamma_j} (\partial_{\nu_j} u_n)(h_j\cdot \nabla \overline u_n)d\Gamma_j\bigg\}=
\end{equation}
$$
a\int_{(\Gamma_j\setminus  \gamma_j)\cap \Gamma}(\psi_j m_j\cdot\nu_j)|\partial_{\nu_j} u_n|^2 d \Gamma_j \leqslant 0.
$$
Inserting \eqref{bounterms.2} into \eqref{pw4.2}, we get 
\begin{align} \label{pw44.2} 
\mbox{Re} \bigg\{2a\int_{\Omega_j}\Delta u_n(h_j\cdot\nabla \overline{u}_n)\bigg\} &\leq
-2a	\mbox{Re}\bigg\{\displaystyle{\sum_{i,k=1}^{N}}\int_{\Omega_j}\partial_ih_j^k\partial_i u_n \partial_k u_ndx\bigg\} \nonumber \\
&+a \int_{\Omega_j}({\rm{div}} h_j) | \nabla u_n | ^2 dx. 
\end{align}
iv) \textbf{The main estimation.}
Inserting equations \eqref{pwl.2}, \eqref{pwa.2}, \eqref{pw5.2} and \eqref{pw44.2}  into (\ref{pw15.2}) and using the fact that $\psi_j=0$ on $Q_1$, we get
\begin{eqnarray}
\begin{array}{l}\label{pw7.2}
\displaystyle\int_{\Omega_j \setminus (Q_1 \cap \Omega_j)}\bigg[ {\rm{div}} (\psi_j m_j)(|\beta_n u_n|^2-a|\nabla u_n|^2)dx + 2a \displaystyle {\sum_{i,k=1}^{N}}\partial_i(\psi_jm_j^k)\partial_i u_n \partial_k u_n\bigg]dx \\ \\
\displaystyle
+2i\int_{\Omega_j \setminus (Q_{1} \cap {\Omega_j})}\beta_n b y_n(\psi_j m_j \cdot\nabla \overline u_n)dx \leqslant o(1).    
\end{array} \nonumber
\end{eqnarray}
Thus, summing over $j$ and using the fact that $ \psi_j =1$ on $\overline{ \Omega_j } \setminus Q_2$, we get
\begin{equation} \label{pw18.2}
\displaystyle N  \int_{\Omega \setminus (Q_2 \cap \Omega) } |\beta_n u_n|^2dx + (2-N)a\int_{\Omega \setminus (Q_2 \cap \Omega)}|\nabla u_n|^2dx 
\end{equation}
$$
+ 
 2\mbox{Re}\bigg\{i\displaystyle \sum_{j=1}^{J}\int_{\Omega_j\setminus (Q_1 \cap \Omega_j)}\beta_n b y_n( \psi_j m_j\cdot\nabla \overline  u_n)dx \bigg \}  $$
$$
\leqslant \displaystyle-\sum_{j=1}^{J}\int_{ Q_2  \cap\Omega_j  }\bigg[ {\rm{div}} (\psi_j m_j)(|\beta_n u_n|^2-a|\nabla u_n|^2)dx + 2a \displaystyle {\sum_{i,k=1}^{N}}\partial_i(\psi_jm_j^k)\partial_i u_n\partial_k u_n\bigg]dx +o(1).
$$
Using \eqref{E33.2} and \eqref{E10.2}, we deduce
\begin{align}\label{pw19.2}
&\displaystyle-\sum_{j=1}^{J}\int_{Q_2 \cap \Omega_j}\bigg[ {\rm{div}} (\psi_j m_j)(|\beta_n u_n|^2-a|\nabla u_n|^2)dx \nonumber \\
&+ 2 a\displaystyle {\sum_{i,k=1}^{N}}\partial_i(\psi_jm_j^k)\partial_i u_n\partial_k u_n\bigg]dx 
= o(1).
\end{align}
Inserting \eqref{pw19.2} in \eqref{pw18.2}, we obtain
\begin{equation}
\label{pwA.2}
\displaystyle N\int_{\Omega\setminus (Q_2\cap \Omega)} |\beta_n u_n|^2dx + (2-N)a\int_{\Omega \setminus (Q_2 \cap \Omega)}|\nabla u_n|^2dx +
\end{equation}
$$
 2 \mbox{Re}\left\{i\displaystyle \sum_{j=1}^{J}\int_{\Omega_j\setminus (Q_1\cap \Omega_j)}\beta_n b y_n( \psi_j m_j.\nabla \overline  u_n)dx\right\}  \leqslant o(1). 
$$
%
%
Under condition (LH3) and the definition of $\psi_j$, we will have  
$$2 \mbox{Re}\left\{i\displaystyle \sum_{j=1}^{J}\int_{\Omega_j\setminus (Q_1\cap \Omega_j)}\beta_n b y_n(\psi_j m_j.\nabla \overline  u_n)dx\right\} = 0.$$
Inserting the previous estimation into \eqref{pwA.2}, we get 
\begin{eqnarray}\label{pwee.2}
\displaystyle N\int_{\Omega\setminus (Q_2\cap \Omega)} |\beta_n u_n|^2dx + (2-N)a\int_{\Omega \setminus (Q_2 \cap \Omega)}|\nabla u_n|^2dx \leq o(1).
\end{eqnarray}
Multiplying (\ref{E5.2}) by $(1-N)\overline{u}_n$. Then integrating on $\Omega$, using Green's formula, the fact that $y_n$ and $\beta_n u_n$ are bounded in $L^2(\Omega)$ and the estimation \eqref{E33.2}, we obtain
\begin{equation}
(1-N) \int_{\Omega}|\beta_n u_n|^2dx -(1-N)a\int_{\Omega}|\nabla u_n|^2dx=o(1). \label{pwe.2}
\end{equation}
Using \eqref{E33.2} and \eqref{E10.2} in \eqref{pwe.2}, we deduce
\begin{equation}
(1-N) \int_{\Omega\setminus(Q_2\cap \Omega)}|\beta_n u_n|^2dx -(1-N)a\int_{\Omega \setminus(Q_2\cap \Omega)}|\nabla u_n|^2dx=o(1). \label{pwf.2}
\end{equation}
Finally, combining \eqref{pwee.2} and \eqref{pwf.2}, we get the following estimate	
\begin{equation*} 
\displaystyle \int_{\Omega \setminus (Q_2\cap \Omega)} \big(a|\nabla u_n|^2 + |\beta_n u_n|^2) dx
=o(1).
\end{equation*}
The proof is thus complete.
{$\quad\square$}\end{proof}
\textbf{Proof of Theorem \ref{IO.2}} It follows from \eqref{E33.2} \eqref{E10.2}, \eqref{ally.2}, \eqref{z.2} and \eqref{pww.2} that $\norm{U_n}=o(1)$ which is a contradiction with \eqref{E1.2}.
Consequently, condition (H2) holds and the energy of system \eqref{e11.2}-\eqref{e13.2} decays exponentially to zero in the weak energy space $D$. The proof is thus complete.{$\quad\square$}
\subsection{Observability and exact controllability}
First, we consider the following homogeneous system associated to \eqref{e11.2}-\eqref{e13.2} for $a\neq1$ by:
\begin{eqnarray}
\psi_{tt}-a\Delta \psi + b(x)\varphi_t&=&  0  \hskip 1.4 cm \mbox{in} \,\,\, \Omega \times \rr_+,\label{eh111.2}\\
\varphi_{tt}- \Delta \varphi - b(x)\psi_t &=&  0  \hskip 1.4 cm \mbox{in}\,\,\, \Omega \times \rr_+,\label{eh112.2}\\
\psi = \varphi &=& 0 \hskip 1.4 cm \mbox{on} \,\, \,  \Gamma\times \rr_+,\label{eh113.2} \\
\psi(\cdot,0) = \psi_0, \psi_t (\cdot,0) = \psi_1, \varphi(\cdot,0) &=& \varphi_0, \varphi_t (\cdot,0) = \varphi_1 \hskip 1.4 cm \mbox{in} \,\, \,  \Omega.\label{eh113.2bis}
\end{eqnarray}
Let $ \Phi=(\psi,\psi_t, \varphi, \varphi_t)$ be a regular solution of system \eqref{eh11.2}-\eqref{eh13.2}, its associated total energy is given by:  
\begin{equation}
E_m(t)=\frac{1}{2} \left(a \norm{\nabla \psi}^2_{L^2(\Omega)} + \norm{\psi_t}^2_{L^2(\Omega)} + \norm{\varphi_t}_{H^{-1}(\Omega)} + \norm{y}_{L^2(\Omega)}\right). \label{energy1.2}
\end{equation}
A direct computation gives 
\begin{equation}
\frac{d}{dt}E_m(t)=0.
\end{equation}
Thus, system \eqref{eh111.2}-\eqref{eh113.2} is conservative in the sense that its energy $E_m(t)$ is constant. It is also wellposed and admits a unique solution in the energy space $D$.

Now, we establish the direct and indirect inequality given by the following theorem: 	
\begin{theorem}	
Let $0<a\neq1$. Assume that conditions {\rm (LH1)} and {\rm (LH2)} hold. Assume also that $\omega_{c_+}$  satisfies the {\rm PMGC}, $\omega_b$ satisfies {\rm{GCC}} condition and {\rm (LH3)} and $b,\, c \in L^\infty(\Omega)$. Then there exists a time $T_0$ such that for all $T>T_0$, there exist two constants $C_1>0$, $C_2>0$ such that the solution of system \eqref{eh111.2}-\eqref{eh113.2} satisfies the following observability inequalities:
\begin{equation}\label{IO.1}
C_1\norm{\Phi_0}^2_D \leq \int_0^T\int_{\Omega} c(x)|\psi_t|^2dxdt \leq C_2\norm{\Phi_0}^2_Ddx,
\end{equation}
for all $\Phi_0 = (\psi_0, \psi_1,\phi_0,\phi_1) \in D$.
\end{theorem}
\begin{proof}
The direct inequality follows from the definition of the total energy for all $T>0$. While the proof of the inverse inequality is a direct consequence of Proposition 2 of Haraux in \cite{Haraux1989} for which the exponentially stability \eqref{ExpSta.2} implies the existence of a time $T_0>0$ such that for all $T>T_0$ there exist two constants $C_1>0$ and $C_2>0$ such that \eqref{IO.1} holds. 
{$\quad\square$}\end{proof}
It is well known that observality of the homogeneous system \eqref{eh111.2}-\eqref{eh113.2} implies the exact controllability of the system . 

\textbf{Acknowledgments}  

The authors are grateful to the anonymous referees and the editor  for their valuable comments and useful suggestions.

The authors thanks professor Kais Ammari for their valuable discussions and comments.

Amina Mortada and Chiraz Kassem would like to thank the AUF agency for its support in the framework of the PCSI project untitled {\it Theoretical and Numerical Study of Some Mathematical Problems and Applications}

Ali Wehbe would like to thank the CNRS and the LAMA laboratory of Mathematics of the Université Savoie Mont Blanc for their supports. 

\end{document}